\documentclass[11pt,  reqno]{amsart}

\usepackage{amsmath,amssymb,amscd,amsthm,amsxtra, esint}

\allowdisplaybreaks[2]

\newtheorem{theorem}{Theorem} [section]

\newtheorem{lemma}[theorem]{Lemma}
\newtheorem{proposition}[theorem]{Proposition}
\newtheorem{remark}[theorem]{Remark}

\DeclareMathOperator*{\supp}{supp}
\DeclareMathOperator{\med}{med}

\newcommand{\noi}{\noindent}

\newcommand{\R}{\mathbb{R}}

\newcommand{\T}{\mathbb{T}}

\newcommand{\N}{\mathcal{N}}
\newcommand{\RR}{\mathcal{R}}
\newcommand{\D}{\mathcal{D}}
\newcommand{\HH}{\mathcal{H}}

\newcommand{\al}{\alpha}
\newcommand{\dl}{\delta}
\newcommand{\Dl}{\Delta}
\newcommand{\eps}{\varepsilon}

\newcommand{\g}{\gamma}
\newcommand{\G}{\Gamma}

\newcommand{\ft}{\widehat}
\newcommand{\wt}{\widetilde}
\newcommand{\cj}{\overline}
\newcommand{\dx}{\partial_x}
\newcommand{\dt}{\partial_t}
\newcommand{\dd}{\partial}

\newcommand{\LRA}{\Longrightarrow}

\newcommand{\jb}[1]
{\langle #1 \rangle}

\numberwithin{equation}{section}
\numberwithin{theorem}{section}

\begin{document}

 
\title
[ Normal forms and Upside-down $I$-method]
{\bf A Remark on Normal Forms and \\the ``Upside-down'' $I$-method
for periodic NLS:\\ Growth of Higher Sobolev Norms}

\author{James Colliander, Soonsik Kwon, and Tadahiro Oh}

\address{James Colliander\\
Department of Mathematics\\
University of Toronto\\
40 St. George St, Rm 6290,
Toronto, ON M5S 2E4, Canada}
\thanks{J.C. is supported 
in part by NSERC grant RGP250233-07}

\email{colliand@math.toronto.edu}

\address{
Soonsik Kwon\\
Department of Mathematical Sciences\\
Korea Advanced Institute of Science and Technology\\
335 Gwahangno (373-1 Guseong-dong) \\
Yuseong-gu, Daejeon 305-701, Republic of Korea}
\thanks{S.K. is supported in part by NRF 2010-0024017}

\email{soonsikk@kaist.edu}

\address{Tadahiro Oh\\
Department of Mathematics\\
Princeton University\\
Fine Hall\\ Washington Road\\
Princeton NJ 08544-1000 USA} 

\email{hirooh@math.princeton.edu}



\begin{abstract}
We study growth of higher Sobolev norms of solutions to 
the one-dimensional periodic nonlinear Schr\"odinger equation (NLS).
By a combination of the normal form reduction and the {\it upside-down} $I$-method,
we establish
\[\|u(t)\|_{H^s} \lesssim (1+|t|)^{\al (s-1)+}\]

\noi
with $\al = 1$ for a general power nonlinearity.
In the quintic case, we obtain the above estimate with $\al = 1/2$ via the space-time estimate 
due to Bourgain \cite{BO3, BO4}.
In the cubic case, we concretely compute the terms arising in the first few steps
of the normal form reduction and prove the above estimate with $\al = 4/9$.
These results improve the previously known results (except for the quintic case.)
In Appendix, we also show how Bourgain's idea in \cite{BO3} on the normal form reduction 
for the quintic nonlinearity
can be applied to other powers.
\end{abstract}

\maketitle

\tableofcontents

\section{Introduction}
We consider 
the periodic defocusing nonlinear Schr\"odinger equation (NLS):
\begin{equation} \label{NLS1}
\begin{cases}
i u_t - u_{xx} + |u|^{2p} u = 0\\
u\big|_{t = 0} = u_0 \in H^s (\T),
\end{cases}
\quad (x, t) \in \T\times \R
\end{equation}

\noi
where $\T = \R / 2 \pi \mathbb{Z}$, $p \in \mathbb{N}$, $s > 1$.
NLS \eqref{NLS1} is a Hamiltonian PDE with Hamiltonian:
\begin{equation} \label{Hamil}
H(u) = \frac{1}{2}\int_\T |u_x|^2 + \frac{1}{2p+2} \int_\T |u|^{2p+2}.
\end{equation}

\noi
Indeed, \eqref{NLS1} can be written as
\begin{equation} \label{NLS2}
u_t = i \frac{\dd H} {\dd \bar{u}}.
\end{equation}

\noi
Recall that \eqref{NLS1} also conserves the $L^2$-norm
and the momentum $P(u) = i \int_\T u \cj{u}_x$.
Moreover, the cubic NLS ($p = 1$) is known to be completely integrable \cite{ZS}
in the sense
that it enjoys the Lax pair structure and so infinitely many conservation laws.
For $p \geq 2$, the $L^2$-norm, the momentum,  and the Hamiltonian
are the only known conservation laws.

In \cite{BO1}, Bourgain proved local well-posedness of \eqref{NLS1}
\begin{itemize}
\item in $L^2(\T)$ for the cubic NLS ($p = 1$),
\item in $H^s(\T)$, $s > 0$, for the quintic NLS ($p = 2$),
\item in $H^s(\T)$, $ s >  \frac{1}{2}-\frac{1}{p}$, for $p \geq 3$.
\end{itemize}

\noi
Hence, \eqref{NLS1} is globally well-posed in $H^1(\T)$
for any $p \in \mathbb{N}$,
since the conservation of
the $L^2$-norm and the Hamiltonian
yields an a priori global-in-time bound on the $H^1$-norm of solutions.
However, 
except for the cubic case ($p = 1$),
there is no a priori upperbound on the $H^s$-norm for $s> 1$.

In this paper, we study growth of higher Sobolev norms $\|u(t)\|_{H^s}$, $s > 1$,  of solutions to \eqref{NLS1}.
By iterating the local theory, we easily obtain an exponential bound
\[ \|u(t)\|_{H^s} \leq C_1e^{C_2|t|},\]

\noi
where $C_1$ and $C_2$ depend only on $s$, $p$, and $u_0$.
This exponential bound is not satisfactory at all.
Polynomial bounds were then obtained 
in Bourgain \cite{BO2}, Staffilani \cite{STA}.
The basic idea is to establish an improved iteration bound:
\[ \|u(t+\tau)\|_{H^s} \leq \|u (t) \|_{H^s} +C \|u(t)\|_{H^s}^{1-\dl}\]

\noi
for all $t \in \R$, 
with some $\dl = \dl (s, p) \in (0, 1)$,
where $\tau$ and $C$ depend on $s, p$, and $u_0$.
This in turn implies
\begin{equation} \label{BOUND} 
\|u(t)\|_{H^s} \leq C (1+|t|)^{\frac{1}{\dl}},
\end{equation}

\noi
where $C = C(s, p, u_0)$.
Fourier multiplier method was used in \cite{BO2},  
 and careful multilinear analysis was performed in \cite{STA}.
(The only result in \cite{BO2, STA} for the one-dimensional periodic NLS is 
for the (nonhomogeneous) cubic NLS with $\dl^{-1} = (s-1)+$ in \cite{STA}.)
Then, Sohinger \cite{SO} applied the {\it upside-down} $I$-method (see below)
to study this problem and proved
\eqref{BOUND} with $\dl^{-1} = 2s+$ for $p\geq 2$
and with $\dl^{-1} = \frac{1}{2}s+$ for $p = 1$.\footnote{
Note the presence of $s$ in place of $s-1$ unlike other results.  
See Remark \ref{REM:1}.}

In the appendix of \cite{BO3}, 
Bourgain applied the normal form reduction to the quintic NLS 
and obtained a growth bound;
if $u$ is a global solution to the quintic NLS \eqref{NLS1} with $p = 2$, 
then we have
\begin{equation} \label{bd1}
\|u(t)\|_{H^s} \lesssim_{s, p, u_0}(1+|t|)^{\frac{1}{2}(s-1)+}
\end{equation}

\noi
for $s >1$.\footnote{We use $A \lesssim$ B to denote an estimate of the form $A \leq  CB$ for some $C > 0$.
Similarly, we use $A \sim B$ to denote $A\lesssim B$ and $B \lesssim A$.
In \eqref{bd1}, the expression $\lesssim_{s, p, u_0}$ shows that the implicit constant $C$ depends on $s, p$, and $u_0$.
In the following, we omit such subscripts when there is no confusion.}
His idea can be applied to other powers, which yields
\begin{theorem}\label{thm1}
Fix $ s> 1$. 
Given $u_0 \in H^s(\T)$,
let
 $u$ be the global solution to \eqref{NLS1}
 with initial condition $u_0$.
 
\begin{itemize} 
\item[\textup{(a)}] Let $p = 1, 2$.
Then, the a priori bound \eqref{bd1} holds.

\item[\textup{(b)}] Let $p \geq 3$.
Then, the following a priori bound holds:
\begin{equation} \label{bd2}
\|u(t)\|_{H^s} \lesssim (1+|t|)^{2(s-1)+}.
\end{equation}
\end{itemize}
\end{theorem}

\noi
Note that both \eqref{bd1} and \eqref{bd2} provide slightly better estimates 
than those in \cite{SO}.
For the cubic ($p=1$) case, 
there are uniform bounds on Sobolev norms due to the complete integrability.
Our interest in this article is to establish an a priori bound without using such a structure 
in an explicit manner.

\medskip

Consider the Hamiltonian corresponding to \eqref{NLS1} in the frequency space:\footnote{In the following, we often drop constants,
when they do not play an important role.}
\begin{align} \label{Hamil2}
H(q) & = H(q, \bar{q}) = \sum_n n^2 |q_n|^2
+ \sum_{n_1 - n_2 + \cdots - n_{2p+2}=0} q_{n_1} \bar{q}_{n_2} \cdots q_{n_{2p+1}}\bar{q}_{n_{2p+2}}\\
& =: H_0(q) + H_1(q), \notag
\end{align}

\noi
where $q_n= \ft{q}(n)$.
Normal form reduction is a sequence of phase space transformations,
transforming the nonlinear part $H_1(q)$ of the Hamiltonian
into expressions involving only ``nearly-resonant'' monomials
for the form
\begin{equation} \label{MONO}
 q_{n_1} \bar{q}_{n_2} \cdots q_{n_{2r-1}}\bar{q}_{n_{2r}}, \qquad r \geq p +1, 
\end{equation}

\noi
where
\begin{equation}
n_1 - n_2 + \cdots + n_{2r-1} - n_{2r} = 0
\end{equation}

\noi
and
\begin{equation}
|n_1^2 - n_2^2 + \cdots + n_{2r-1}^2 - n_{2r}^2| <K
\end{equation}

\noi
for some large $K>0$,
(plus a non-resonant error, which needs to be estimated in a suitable topology.)
By choosing $K = T^{-\dl}$ for some small $\dl>0$,
Bourgain \cite{BO3} applied the normal form reduction with the $L^6$-Strichartz estimate 
(see \eqref{L6} and \eqref{L666} below)
and established \eqref{bd1} for $|t| \leq T$.\footnote{In \eqref{bd1}, 
the implicit constant is independent of $T$, 
and hence the bound \eqref{bd1} holds for all $t\in \R$.}
In Appendix, we briefly discuss how his idea can be applied to other powers.

\medskip
In order to improve Theorem \ref{thm1},
we combine this normal form reduction with the {\it upside-down} $I$-method.
For $s > 1$, let $\mathcal{D}$ be the Fourier multiplier operator given by
the multiplier $m:\mathbb{Z} \to \R$,
where
\begin{equation} \label{DD1}
m(n) = \begin{cases}
1, & |n| \leq N\\
\big(\frac{|n|}{N}\big)^{s-1}, & |n| > N.
\end{cases}
\end{equation}

\noi
The operator $\D$ is basically a differentiation operator of order $s -1$.
Moreover, it satisfies
\begin{equation} \label{DD}
\|\D q \|_{H^1} \leq \|q\|_{H^s} \leq N^{s-1} \|\D q\|_{H^1}.
\end{equation}

\noi
The upside down $I$-method first appeared in \cite{CKSTT1}
(in the low regularity setting.)
The growth of Sobolev norm is related to 
the low-to-high frequency cascade,
and the (upside-down) $I$ method is a suitable tool to study such a phenomenon.
As a result, we obtain the following improvement.

\begin{theorem}\label{thm2}
Fix $ s> 1$. Given $u_0 \in H^s(\T)$,
let
 $u$ be the global solution to \eqref{NLS1}
 with initial condition $u_0$.

\begin{itemize} 
\item[\textup{(a)}] 
Let $p \geq 3$.
Then, we have
\begin{equation} \label{bd3}
\|u(t)\|_{H^s} \lesssim (1+|t|)^{(s-1)+}.
\end{equation}

\item[\textup{(b)}] 
Let $p =2$. Then,  the a priori bound \eqref{bd1} holds.

\item[\textup{(c)}] 
Let $p =1$. Then, the following a priori bound holds:
\begin{equation} \label{bd4}
\|u(t)\|_{H^s} \lesssim (1+|t|)^{\frac{4}{9}(s-1)+}.
\end{equation}

\end{itemize}

\end{theorem}

\begin{remark}\label{REM:1} \rm
In \cite{SO}, Sohinger defined $\D$ to be a differentiation of order $s$
and proved an estimate on $\|\D u(t)\|_{L^2}$,
i.e. his argument is based on almost conservation of the $L^2$-norm.
However, it seems that by using $\D$ as in \eqref{DD1}
with almost conservation of the Hamiltonian ($\sim H^1$-norm), 
one can obtain the results in \cite{SO}, but with $s-1$ in place of $s$.
\end{remark}

Our argument is closely related to that by Bourgain in \cite{BO4},
where he combined the normal form reduction and the $I$-method
to study global well-posedness of the defocusing quintic NLS on $\T$.
There are two main steps in the proof of Theorem \ref{thm2}.
First, we apply the normal form reduction to the Hamiltonian $H$ in \eqref{Hamil2}
and obtain a new Hamiltonian $\HH = H\circ \G$
with a certain symplectic transformation $\G$
so that the transformed Hamiltonian $\HH$ is 
of the form \[\HH(q) = H_0(q) + \N(q),\]

\noi
where $\N$ consists of nearly-resonant terms (plus ``small'' error.)
Our choice of the symplectic transformation $\G$
satisfies
$\| \G q\|_{L^2} = \|q\|_{L^2}$
and $\| \G q\|_{H^1} \sim \|q\|_{H^1}$.
Recall from \cite{BO3} that $K = T^{-\dl}$ for Theorem \ref{thm1}.
For Theorem \ref{thm2}, we choose $K = N^\dl$ for some small $\dl > 0$,
and then choose $N$ in terms of $T$ as in the usual (upside-down) $I$-method.

After performing the normal form reduction, 
we apply the upside-down $I$-method
to the transformed Hamiltonian $\HH$.
Suppose that $q(t)$ satisfies the Hamiltonian flow of $\mathcal{H}$, i.e.
\[ q_t = i \frac{\dd \mathcal{H}}{\dd q}.\]

\noi
Then, differentiating in time as in \cite{BO4}, we obtain 
\begin{align}
\frac{d}{dt} \HH(\D q) & = \frac{\dd \HH}{\dd q}(\D q)\cdot \D q_t 
+ \frac{\dd \HH}{\dd \bar{q}}(\D q)\cdot \cj{\D q}_t  \notag \\
& = i \sum_n m(n)^2 n^2 \bigg( \bar{q}_n \frac{\dd \N}{\dd \bar{q}_n}(q)
 - q_n \frac{\dd \N}{\dd q_n}(q) \bigg)  \label{HH1} \\
& + i \sum_n m(n) n^2 \bigg( q_n \frac{\dd \N}{\dd q_n}(\D q)
- \bar{q}_n \frac{\dd \N}{\dd \bar{q}_n}(\D q) \bigg)  \label{HH2}\\
& + i \sum_n m(n) \bigg(  \frac{\dd \N}{\dd q_n}(\D q) \frac{\dd \N}{\dd \bar{q}_n}( q)
- \frac{\dd \N}{\dd q_n}( q) \frac{\dd \N}{\dd \bar{q}_n}( \D q) \bigg).  \label{HH3}
\end{align}

\noi
As noted in \cite{BO4}, we have
$\eqref{HH1} + \eqref{HH2} = 0$ and
$\eqref{HH3} = 0$ if $\supp q \subset [-N, N]$.
Hence, we assume that 
\begin{equation} \label{MAX}
\max (|n_1|, \dots, |n_{2r}|) > N
\end{equation}

\noi
for the monomials of the form \eqref{MONO}.
Then, we prove Theorem \ref{thm2} (a)
by estimating the contributions from \eqref{HH1}--\eqref{HH3}.
When $p \leq 2$, we obtain an improvement from the space-time estimate
by Bourgain \cite{BO3, BO4}.
See \eqref{L666} below.
Finally, for the cubic nonlinearity ($p = 1$),
we concretely compute the terms arising in the first few steps
of the normal form reduction,
and show that these terms (as well as the higher order terms) satisfy better estimates.
A further improvement may be achieved by computing more terms in the normal form reduction.
However, the actual computation becomes very cumbersome and we do not pursue this direction
any further in this article.
See Gr\'ebert-Kappeler-P\"oschel \cite{GKP}
for the normal form theory of the defocusing cubic NLS,
based on the integrability of the equation.

%

For the non-periodic cubic NLS, 
Sohinger \cite{SO2} used the a priori bound on the $H^{k}$-norm, $k \in \mathbb{N}$,
and obtained
\[\|u(t)\|_{H^s} \lesssim (1+|t|)^{\{s\}+},\]

\noi
where $\{s\}$ denotes the fractional part of $s >1$.
Note that  such uniform bounds on the $H^k$-norms are results of
integrability of the equation. See \cite{FT, ZM}.
One could try to prove a similar result in the periodic case.
However, we do not pursue this direction in this article, 
since our focus is to present an analytical method without using the complete integrability 
in an explicit manner.

\medskip

This paper is organized as follows.
In Section 2, we briefly review the theory of the normal form reduction,
and apply it in the NLS context.
In Section 3, we apply the upside-down $I$-method
to the transformed Hamiltonian
and 
prove Theorem \ref{thm2} (a) and (b).
In Section 4, we focus on the cubic NLS.
By explicitly computing the first few steps of the normal form reduction,
we establish improved estimates   
in applying the upside-down $I$-method,
and prove Theorem \ref{thm2} (c).
In Appendix, we discuss Theorem \ref{thm1}
and show how to apply Bourgain's idea \cite{BO3}
for general powers.

\section{Normal Form Reduction} \label{SEC:NORMAL}
\subsection{Introduction}

The normal form reduction involves in eliminating non-resonant parts
of the Hamiltonians by introducing suitable symplectic transformations.
Our goal is to repeat this procedure
so that the transformed Hamiltonian $\HH$ consists of 
the quadratic part $H_0$, the resonant part $\N_0$, and the error $\N_r$.
In the following,  we briefly review the basic procedure of the normal form reduction.
Also, see Kuksin-P\"oschel \cite{KP}, Bourgain \cite{BO3, BO4}, Gr\'ebert \cite{GRE}.

Let 
\begin{equation}\label{Hamil3}
\wt{H}(q, \bar{q})
=  \sum_{n_1 - n_2 + \cdots - n_{2r}=0} c(\bar{n})q_{n_1} \bar{q}_{n_2} \cdots q_{n_{2r-1}}\bar{q}_{n_{2r}}
\end{equation}

\noi
be (a part of) a Hamiltonian obtained at some stage of this process.
Assume that $c(\bar{n}) := c(n_1, \cdots, n_{2r}) \in \R$
and that $c(\bar{n})$ is invariant\footnote{This is satisfied by the initial Hamiltonian \eqref{Hamil2},
and thus is automatically satisfied by all the Hamiltonians appearing in the process.} 
(modulo $\pm$ signs)
under the permutation $n_{2k-1} \leftrightarrow n_{2k}$,
$k = 1, \dots, r$.
Divide $\wt{H}$ into the resonant\footnote{Strictly speaking, $\wt{H}_0$ is only ``nearly resonant''.
However, we refer to it as the ``resonant'' part for simplicity.} part $\wt{H}_0$ and non-resonant part $\wt{H}_1$,
i.e. $\wt{H}_0$ (and $\wt{H}_1$) is the restriction of $\wt{H}$
on $|D(\bar{n})| \leq K$ (and  $|D(\bar{n})| >K$, respectively), 
where $D(\bar{n})$ is defined by 
\begin{equation} \label{RES}
D(\bar{n}) := n_1^2 - n_2^2 + \cdots + n_{2r-1}^2 - n_{2r}^2.
\end{equation}

\noi
We now introduce a symplectic transformation $\G = \G_F$, called the Lie transform,
to eliminate $\wt{H}_1$.
Define a Hamiltonian $F$  ($= \text{``}D^{-1}\wt{H}_1 \text{''}$) by
\begin{equation} \label{Hamil4}
F(q, \bar{q}) = \sum_{\substack{n_1 - n_2 + \cdots - n_{2r}=0\\ |D(\bar{n})|>K}} \frac{c(\bar{n})}{D(\bar{n})}
q_{n_1} \bar{q}_{n_2} \cdots q_{n_{2r-1}}\bar{q}_{n_{2r}}.
\end{equation}

\noi
Then, it is not difficult to see that $F$ satisfies the following homological equation:
\begin{equation} \label{HOMOLOGY}
\{ H_0, F\} = -\wt{H}_1,
\end{equation}

\noi
where $H_0(q) = \sum_n n^2 |q_n|^2$ as in \eqref{Hamil2}
and the Poisson bracket $\{ \cdot, \cdot\}$ is defined by
\begin{equation} \label{Poisson}
\{ H_1, H_2\} = i \sum_n \bigg[ \frac{\dd H_1}{\dd q_n}\frac{\dd H_2}{\dd \bar{q}_n}
- \frac{\dd H_1}{\dd \bar{q}_n}\frac{\dd H_2}{\dd q_n}\bigg].
\end{equation}

\noi
Consider a Hamiltonian flow associated to the Hamiltonian $F$:
\begin{equation} \label{LIE}
 q_t = i \frac{\partial F}{\partial \bar{q}}.
\end{equation}

\noi
Let $\G_t$ denote the flow map generated by $F$ at time $t$.
Then, we define the Lie transform $\G (= \G_F)$ to be the time-1 map $\G_1$.\footnote{Here, we simply assume that 
the flow exists up to time $t = 1$. See Subsection \ref{SUBSEC:2.3}.}
As we see below, the non-resonant part $\wt{H}_1$
is eliminated under $\G$.

Recall the following lemma. See \cite[Lemma 2.8]{GRE}.
\begin{lemma} \label{LEM:LIE}
Let $\G_t$ be as above.
Then, for a smooth function $G$, 
 we have
\[\frac{d}{dt} (G\circ \G_t) = \{ G, F\} \circ \G_t.\]
\end{lemma}

\begin{proof}
By Chain Rule, we have
\begin{align*}
\frac{d}{dt} (G\circ \G_t) 
& =  \frac{\dd G}{\dd q}(q(t)) \cdot q_t + \frac{\dd G}{\dd \bar{q}}(q(t)) \cdot \bar{q}_t \\
& = i \frac{\dd G}{\dd q} \cdot \frac{\partial F}{\partial \bar{q}}
- i \frac{\dd G}{\dd \bar{q}} \cdot \frac{\partial F}{\partial q}
= \{G(q(t)), F(q(t))\}
\end{align*}

\noi
since $\cj{ \partial F / \partial \bar{q}} = \partial F / \partial q$.
\end{proof}

\noi
By the Taylor series expansion of $G\circ \G_t$ centered at $t = 0$, 
we obtain
\begin{equation} \label{TAYLOR}
G\circ \G = \sum_{k = 0}^\infty \frac{1}{k!} \{G, F\}^{(k)},
\end{equation}

\noi
where $\{G, F\}^{(k)}$ denotes the $k$-fold Poisson bracket of $G$ with $F$,
i.e.
\[\{G, F\}^{(k)} := \{ \cdots \{ G, 
\underbrace{ F\}, F\}, \cdots, F\} }_{k \text{ times}} \]

\noi
and $\{G, F\}^{(0)} = G$.

Suppose that we start with a Hamiltonian $H = H_0 + \wt{H}$,
where $H_0$ is as in \eqref{Hamil2} and $\wt{H}$ is as in \eqref{Hamil3}.
From \eqref{TAYLOR} and \eqref{HOMOLOGY}, the transformed Hamiltonian $H' = H\circ \G$ is given by
\begin{align*}
H' & = H\circ \G = H_0 \circ \G + \wt{H}_0 \circ \G + \wt{H}_1 \circ \G \\
& = H_0 + \wt{H}_0 + \wt{H}_1
+ \{ H_0, F\} +  \{ \wt{H}_0, F\} + \{ \wt{H}_1, F\} + \text{h.o.t.}\\
& = H_0 + \wt{H}_0 
 +  \{ \wt{H}_0, F\} + \{ \wt{H}_1, F\} + \text{h.o.t}.,
\end{align*}

\noi
where ``h.o.t.''~stands for higher order terms.
Hence, we have eliminated the non-resonant part $\wt{H}_1$ by 
the Lie transform $\G$.
Then, we define the resonant part $\wt{H}'_0$
and the non-resonant part $\wt{H}'_1$ of the new Hamiltonian $H'$
by 
\begin{align*}
\wt{H}'_0 & := \wt{H}_0 + \text{resonant part of } \{ \wt{H}_0, F\} + \{ \wt{H}_1, F\} + \text{h.o.t.}\\
\wt{H}'_1 & := \text{non-resonant part of }\{ \wt{H}_0, F\} + \{ \wt{H}_1, F\} + \text{h.o.t}.
\end{align*}

\noi
Note that at each step, the lowest degree among the monomials in the non-resonant part
increases at least by two since $\text{deg} \,F \geq 4$.

Lastly, we discuss the regularity of the Lie transform $\G$.
It follows from Sobolev embedding that $\G$ acts boundedly
on bounded subsets of $H^s(\T)$, $s >\frac{1}{2}$.
See \cite{KP}.
Indeed, for $F$ as in \eqref{Hamil4}, 
by H\"older inequality and Sobolev embedding, we have
\begin{align*}
\bigg\|\frac{\partial F}{\partial \bar{q}}\bigg\|_{H^s}
& \lesssim \sup_{\|p\|_{L^2} = 1 }
K^{-1} \sum_{n_1 - n_2 + \cdots + n_{2r-1} - n = 0} 
|c(\bar{n})|  |q_{n_1}| \cdots |q_{n_{2r-1}}| \cdot \jb{n}^s |p_{n}|\\
& \lesssim \sup_{\|p\|_{L^2} = 1 } \|p\|_{L^2} \|q\|_{H^s} \|q\|_{H^{\frac{1}{2}+}}^{2r-2}
\leq \|q\|_{H^{s}}^{2r-1},
\end{align*}

\noi
where we used the fact that $\jb{n} \lesssim \max (\jb{n_1}, \dots, \jb{n_{2r-1}})$
in the second inequality.
This is sufficient for our purpose 
since we take the phase space 
to be $H^1(\T)$ for $\D q$
(and $H^s(\T)$, $s >1$,
for our initial data $q$.)
See \cite{BO3, BO4} for the boundedness of $\G$ in $H^\eps (\T)$, $\eps > 0$,
for the quintic case.

\subsection{Normal form reduction}

In this subsection, we actually implement the normal form reduction
to the Hamiltonian $H$ in \eqref{Hamil2} corresponding to NLS \eqref{NLS1}.
Our goal is the following; 
by a finite\footnote{We repeat the process only finitely many times.
In particular, the degree $2r$ of monomials  is finite.} sequence of Lie transforms, we transform $H$ 
into a Hamiltonian of the form
\begin{equation} \label{Hamil5}
\HH (q) = H_0 (q) + \N_0 (q) + \N_r (q) ,
\end{equation}

\noi
where $H_0$ is the quadratic part, 
$\N_0$ is the resonant part $\N_0$, and $\N_r$  is ``small'' error.
We assume that $q = \{q_n\}_{n\in \mathbb{Z}}$
satisfies the following $L^2$- and $H^1$-bounds:
\begin{align}
\label{L2}
& \|q\|_{L^2} \leq C_1,\\
\label{H1}
& \|q\|_{H^1} \leq C_2.
\end{align}

\noi
In Section \ref{SEC:3},
we use the result of this section with
 $\D q \in H^1$ (for given $q \in H^s$, $s>1$)
as the phase space element
in place of $q$ in \eqref{L2} and \eqref{H1}. 

First, we need to define the ``norm'' $\|\, \cdot\, \|$
to measure a size of a (homogeneous) Hamiltonian.
Given a homogeneous multilinear expression
\begin{equation}\label{Hamil6}
\N(q, \bar{q})
=  \sum_{n_1 - n_2 + \cdots - n_{2r}=0} c(\bar{n})q_{n_1} \bar{q}_{n_2} \cdots 
q_{n_{2r-1}}\bar{q}_{n_{2r}},
\end{equation}

\noi
define the ``size'' $\|\N\|$ of $\N$ by 
\begin{equation} \label{NORM}
\|\N\| = \sup_* \sum_n |c(\bar{n})| |q^{(1)}_{n_1}| |q^{(2)}_{n_2}| \cdots |q^{(2r)}_{n_{2r}}|
\end{equation}

\noi
where the supremum is taken over 
factors $q^{(j)}$, $1\leq j \leq 2r$
such that 
\begin{itemize}
\item all factors satisfy \eqref{L2}

\item all except at most two factors also satisfy \eqref{H1}.
\end{itemize}

\noi
i.e. the supremum is taken over all the factors, allowing two {\it exceptional} ones.
See \cite{BO4} for a similar definition of a norm
on homogeneous multilinear expressions.
Like (3.6) in \cite{BO4}, we obtain the following proposition
on closure of the Poisson bracket under this norm.

\begin{proposition} \label{PROP:POIS1}
Let $H_1$ and $H_2$ be Hamiltonians of the form \eqref{Hamil6}.
Then, we have
\begin{equation}
\|\{H_1, H_2\}\|\lesssim \|H_1\|\|H_2\|.
\end{equation}

\end{proposition}

We need the following lemma to prove Proposition \ref{PROP:POIS1}.

\begin{lemma} \label{LEM:POIS2}
Let $H$  be a Hamiltonian of the form \eqref{Hamil6}.

\noi
\textup{(a)} If all the factors of $\dd H/\dd \bar{q}$ satisfy both \eqref{L2} and \eqref{H1}, 
then we have
\begin{align}
\label{H11} & \bigg\| \frac{\dd H}{\dd \bar{q}}\bigg\|_{H^1} \leq \|H\|.
\end{align}

\noi
\textup{(b)}
If all the factors of $\dd H/\dd \bar{q}$ satisfy \eqref{L2}
and all except at most one satisfy \eqref{H1}, 
then we have
\begin{align}
\label{L22} & \bigg\| \frac{\dd H}{\dd \bar{q}}\bigg\|_{L^2} \leq \|H\|.
\end{align}
\end{lemma}

\begin{proof}
(a) Without loss of generality, assume $|n| \lesssim |n_1|$ since 
$n = n_1 - n_2 + \cdots + n_{2r-1}$.
By duality, we have
\begin{align*}
\text{LHS of } \eqref{H11}
& \lesssim 
\sup_{\|p\|_{L^2} = 1}
\sum_{n_1 - \cdots +n_{2r-1}- n=0} |c(\bar{n})| \big(|n_1| |q_{n_1}| \big)
|\bar{q}_{n_2}| \cdots |q_{n_{2r-1}}| |\bar{p}_{n}|\\
& \leq \|H\|,
\end{align*}

\noi
since $\big\||n_1| q_{n_1}\big\|_{l^2}$, $\|\bar{p}_{n}\|_{l^2} \lesssim 1$,
i.e. all the factors satisfy \eqref{L2}
and all, except for $|n_1| q_{n_1}$ and $\bar{p}_{n}$, satisfy \eqref{H1}.
Part (b) follows in a similar manner.
\end{proof}

\begin{proof}[Proof of Proposition \ref{PROP:POIS1}]
It suffices to prove
\begin{align} \label{H111}
\bigg\| \sum_n \frac{\dd H_1}{\dd q_n} \frac{\dd H_2}{\dd \bar{q}_n} \bigg\|
\lesssim \|H_1\| \|H_2\|.
\end{align}

\noi
There are three cases, depending on the location of
the two {\it exceptional} factors.

First, suppose that both appear in $\dd H_1/\dd q_n$.
By duality, we have
\begin{align*}
\bigg\| \frac{1}{\jb{n}} \frac{\dd H_1}{\dd q_n} \bigg\|_{l^2_n}
\leq \sup_{\|p\|_{L^2} = 1} \sum_{n_1 - \cdots  +n_{2r-1}- n=0} |c_1(\bar{n})| |q_{n_1}| 
|\bar{q}_{n_2}| \cdots |q_{n_{2r-1}}| \big(\jb{n}^{-1}|\bar{p}_{n}|\big)
\leq \|H_1\|,
\end{align*}

\noi
(where $\jb{n} := 1+|n|$), 
since $\jb{n}^{-1}\bar{p}_{n}$ with $\|p\|_{L^2} = 1$ satisfies both \eqref{L2} and \eqref{H1}.
Hence, from Lemma \ref{LEM:POIS2} (a), we have
\begin{align*}
\bigg|\sum_n \frac{\dd H_1}{\dd q_n}\frac{\dd H_2}{\dd \bar{q}_n}\bigg|
\leq \bigg\| \frac{1}{\jb{n}} \frac{\dd H_1}{\dd q_n} \bigg\|_{l^2_n}  \bigg\|\frac{\dd H_2}{\dd \bar{q}}\bigg\|_{H^1}
\leq \|H_1\|\|H_2\|.
\end{align*}

\noi
The same argument holds when both exceptional factors appear in $\dd H_2/\dd \bar{q}_n$.
Finally, suppose
that exactly one exceptional factor appears
in each of $\dd H_1/\dd q_n$ and $\dd H_2/\dd \bar{q}_n$.
Then, \eqref{H111} follows from Cauchy-Schwarz inequality
and Lemma \ref{LEM:POIS2} (b). 
\end{proof}

\medskip
Now, we inductively iterate the steps of the normal form reduction,
assuming \eqref{L2} and \eqref{H1}. 
For fixed $N$ (to be chosen in terms of $T$ in the next section),
we set $K = N^{\dl}$ for some small $\dl > 0$.
(Recall that we have $K = T^{\dl}$ in \cite{BO3}.)
{\it Assume} that at some stage of the process, 
the Hamiltonian is of the form
\begin{equation} \label{Hamil7}
\HH (q) = \sum_n n^2 |q_n|^2 + \N_0 (q) + \N_1(q) + \N_r(q) ,
\end{equation}

\noi
where 
the monomials in the resonant part $\N_0$
satisfy
\begin{equation} \label{N1}
|D(\bar{n}) | \leq N^\dl
\end{equation}

\noi
for some small $\dl > 0$, (where $D(\bar{n})$ is as in \eqref{RES}),
the monomials in the non-resonant part $\N_1$
satisfy
\begin{equation} \label{N2}
|D(\bar{n}) | > N^\dl,
\end{equation}

\noi
and 
the remainder part $\N_r$ satisfies
\begin{align} \label{N3}
\|\N_r\| < N^{-C}
\end{align}

\noi
for some large $C>0$.
Moreover, we have
\begin{equation} \label{N4}
\|\N_0\|, \, \|\N_1\| \lesssim 1. 
\end{equation}

Clearly, the initial Hamiltonian in \eqref{Hamil2}
satisfies the above conditions.
Note that there is no remainder part at this stage,
i.e. $\N_r = 0$.
By Sobolev embedding along with \eqref{L2} and \eqref{H1},
we have
\begin{align*}
|H_1(q)| 
& = \bigg|\sum_{n_1 - n_2 + \cdots - n_{2p+2}=0} q_{n_1} \bar{q}_{n_2} \cdots q_{n_{2p+1}}\bar{q}_{n_{2p+2}}\bigg|\\
& \leq \| q\|_{L^2}^2 \|q\|_{L^\infty}^{2p}
\leq \| q\|_{L^2}^2 \|q\|_{H^{\frac{1}{2}+}}^{2p}
\lesssim 1.
\end{align*}

\noi
i.e. $\|H_1\| \lesssim 1$.
Hence, the resonant and non-resonant parts of $H_1$ satisfy \eqref{N4}.

\medskip
Assume \eqref{Hamil7}.
Suppose that (the collection of monomials of the lowest degree in) 
the non-resonant part $\N_1$ is given by
\[
\sum_{\substack{n_1 - n_2 + \cdots - n_{2r}=0\\ |D(\bar{n})|>N^\dl}} c(\bar{n})
q_{n_1} \bar{q}_{n_2} \cdots q_{n_{2r-1}}\bar{q}_{n_{2r}}.
\]

\noi
As in the previous subsection,
define $F$ by
\begin{equation} \label{Hamil8}
F(q, \bar{q}) = \sum_{\substack{n_1 - n_2 + \cdots - n_{2r}=0\\ |D(\bar{n})|>N^\dl}} \frac{c(\bar{n})}{D(\bar{n})}
q_{n_1} \bar{q}_{n_2} \cdots q_{n_{2r-1}}\bar{q}_{n_{2r}}
\end{equation} 

\noi
so that we have
\begin{equation} \label{HOMOLOGY2}
 \{F, H_0\} = \N_1.
\end{equation}

\noi
Let $\G$ be the Lie transform associated to $F$.
Then, by \eqref{TAYLOR},  we have
\begin{align*}
\HH'  = \HH \circ \G 
 & = H_0 +  \N_0 + \N_1  \\
& \hphantom{Xl} + \{ H_0, F\} 
+ \{ \N_0, F\} + \{ \N_1, F\} +  \text{h.o.t.}\\
& \hphantom{Xl}+ \N_r \circ \G\\
& = H_0 +  \N_0   
 + \{ \N_0, F\} + \{ \N_1, F\} +  \text{h.o.t.}
 + \N_r \circ \G.
\end{align*}

From \eqref{N2} and \eqref{N4}, we have
\begin{equation}\label{F1}
\|F\| \leq N^{-\dl} \|\N_1\| \lesssim N^{-\dl}.
\end{equation}

\noi
From Proposition \ref{PROP:POIS1}, \eqref{N3}, \eqref{F1}, and \eqref{TAYLOR}, 
we see that $\N_r \circ \G$ satisfies \eqref{N3}.
It also follows that
the higher order terms with sufficiently high degrees satisfy \eqref{N3}.

Let $\mathfrak{N}$ denote 
the sum of $\{ \N_0, F\}$, $\{ \N_1, F\}$, and the remaining part of the higher order terms,
i.e.
\[\mathfrak{N} = 
\sum_{k = 1}^M \frac{1}{k!} \{\N_0, \, F\}^{(k)}
+ \sum_{k = 1}^M \frac{1}{k!} \{\N_1, \, F\}^{(k)}
+ \sum_{k = 2}^M \frac{1}{k!} \{H_0, \, F\}^{(k)}\]

\noi
for some $M \in \mathbb{N}$.
From Proposition \ref{PROP:POIS1}, \eqref{Hamil8}, and \eqref{N4}, we have
\[\|\{ \N_0, F\}\| \lesssim N^{-\dl} \|\N_0\| \|\N_1\|
\lesssim N^{-\dl} \|\N_1\|.
\]

\noi
Similarly, we have
$\|\{ \N_0, F\}^{(k)}\|\lesssim N^{-k\dl} \|\N_1\|$.
Then, 
from \eqref{F1} and \eqref{HOMOLOGY2}, we have
\begin{align*}
\|\mathfrak{N}\| \lesssim 
 \|\N_1\|\bigg\{    \sum_{k = 1}^M \frac{1}{k!} N^{-k\dl}
+ \sum_{k = 2}^M \frac{1}{k!} N^{-(k-1)\dl} \bigg\}
\lesssim N^{-\dl}\|\N_1\|. 
\end{align*}

\noi
Now, according to \eqref{N1} and \eqref{N2}, 
divide $\mathfrak{N}$ into its resonant part $\mathfrak{N}_0$
and its non-resonant part $\mathfrak{N}_1$.
Hence, we can write the new Hamiltonian $\HH'$ as
\[ \HH' = H_0 + \N_0' + \N_1' + \N_r'\]

\noi
where $\N_0' := \N_0 + \mathfrak{N}_0$
satisfies \eqref{N4},
$\N_1' := \mathfrak{N}_1$ satisfies 
\begin{equation} \label{F2}
\|\N_1'\| \lesssim N^{-\dl} \|\N_1\|,
\end{equation}

\noi
and $\N_r'$ satisfies \eqref{N3}.
In view of \eqref{F2}, we can hide the non-resonant part
into the remainder part,
by iterating the process sufficiently many times.

Therefore, by a {\it finite} sequence of Lie transforms, 
we have obtained  a new Hamiltonian $\HH$ of the form
\begin{equation} \label{Hamil9}
\HH (q) = \sum_n n^2 |q_n|^2 + \N_0 (q) + \N_r (q) ,
\end{equation}

\noi
where $\|\N_0\| \lesssim 1$
and $\|\N_r\|\lesssim N^{-C}$.

\subsection{$L^2$- and $H^1$-bounds under the Lie transform} \label{SUBSEC:2.3}

Before proceeding with the upside-down $I$-method, 
let us discuss how the conditions \eqref{L2} and \eqref{H1}
are affected under the Lie transform.

Given $F$ as in \eqref{Hamil8},
we define the Lie transform $\G$
to be the time-1 map of \eqref{LIE}.
Denoting by $\G_t$ the flow map of \eqref{LIE} at time $t$,
we have
\begin{equation} \label{LIE2}
\G_t q = q(t) = q(0) +i \int_0^t \frac{\dd F}{ \dd \bar{q}}(q(t')) dt',
\end{equation}

\noi
where $q = q(0)$ and $\G_t q = q(t)$.
Let $M(q) = \|q\|_{L^2}^2 = \sum_n |q_n|^2$.
Then, by Lemma \ref{LEM:LIE}, we have
\[ \frac{d}{dt}M(q(t)) = \{ M(q(t)), F(q(t))\} =0.\]

\noi
Hence, we have $\|\G_t q\|_{L^2} = \|q\|_{L^2}$.
In particular, we obtain 
\begin{equation} \label{L222}
\|\G q\|_{L^2} = \|q\|_{L^2}.
\end{equation}

Now, take the $H^1$-norm in \eqref{LIE2}. 
From Lemma \ref{LEM:POIS2} (a), \eqref{NORM}, and \eqref{N2}, 
we have
\begin{align*}
\|\G_t q\|_{H^1}
& \leq \|q\|_{H^1} 
+ t \sup_{t' \in [0, t]} \bigg\|\frac{\dd F}{ \dd \bar{q}}(q(t')) \bigg\|_{H^1}
\leq \|q\|_{H^1} + C t \sup_{t' \in [0, t]} \|F(q(t'))\|\\
& \leq \|q\|_{H^1} + C' t N^{-\dl} \sup_{t' \in [0, t]} \|q(t')\|_{L^2}^2 \|q(t')\|_{L^\infty}^{2r-2}.
\end{align*}

\noi
By taking the supremum over $t \in [0, 1]$
and by Sobolev embedding
along with interpolation on the $L^2$- and $H^1$-norms, we have
\begin{align*}
\sup_{t \in [0, 1]} \|\G_t q\|_{H^1}
& \leq \|q\|_{H^1} + C' N^{-\dl} \sup_{t \in [0, 1]} \|\G_t q\|_{H^1}^{r-1+}
\end{align*}

\noi
Now, choose $N = N(\g) $ large enough
such that if $X_t \leq 4\g $, then
\begin{equation} \label{F3}
X_t \leq \g + C'' N^{-\dl} X_t^{r-1+} \quad \text{ implies } \quad
X_t \leq 2\g.
\end{equation}

\noi
For our purpose, let $\g = \|q\|_{H^1}$.
By the local theory of \eqref{LIE}, there exists a time $[0, \eps_0]$
such that $X_t := \|\G_t q\|_{H^1} \leq 2\g$ for $t \in [0, \eps_0]$.
In particular, we have $\|\G_{\eps_0} q\|_{H^1} \leq 2\g$.
By the local theory again, 
there exists $\eps > 0$ such that $X_t := \|\G_t q\|_{H^1} \leq 4\g$ for $t \in [0,  \eps_0 + \eps]$.
By \eqref{F3}, we have $\|\G_t q\|_{H^1} \leq 4\g$
for  $t \in [0, \eps_0 + \eps]$.
By iterating the argument with a fixed size of $\eps$,
we obtain
$\|\G_t q\|_{H^1} \leq  2 \|q\|_{H^1}$ for  $t \in [0, 1]$.
By inverting the time, we obtain
\begin{equation} \label{H123}
\|\G q\|_{H^1} \sim \|q\|_{H^1} \sim \|\G^{-1} q\|_{H^1}.
\end{equation}

From \eqref{L222} and \eqref{H123},
we see that the conditions \eqref{L2} and \eqref{H1} are preserved under the Lie transform.

\section{Upside-down $I$-method} \label{SEC:3}
In this section, we estimate the terms \eqref{HH1}, \eqref{HH2}, and \eqref{HH3}
appearing in $\frac{d}{dt} \HH(\D q)$,
where $\HH$ is the Hamiltonian of the form \eqref{Hamil9} obtained in the previous section.
The analysis is very similar to that in \cite{BO4}.
We estimate $|\frac{d}{dt} \HH(\D q)|$
in terms of a negative power of $N$ and then prove Theorem \ref{thm2}.

\subsection{Estimates on \eqref{HH1}, \eqref{HH2}, and \eqref{HH3}} \label{SEC:NONLIN1}
In the following, we assume that $\N$ is of the form \eqref{Hamil6}.
Then, we can rewrite \eqref{HH1} and \eqref{HH2} as follows.
\begin{align} 
\eqref{HH1}
& = - \sum_{n_1 - n_2+ \cdots - n_{2r} = 0}
c(\bar{n}) R(\bar{n})
q_{n_1}\bar{q}_{n_2}\cdots\bar{q}_{n_{2r}} \label{HH4}
\\
\eqref{HH2}
& = \sum_{n_1 - n_2+ \cdots - n_{2r} = 0}
c(\bar{n}) D(\bar{n})
\D q_{n_1} \cj{\D q}_{n_2}\cdots\cj{\D q}_{n_{2r}}, \label{HH5}
\end{align}

\noi
where $D(\bar{n})$ is as in \eqref{RES} and
$R(\bar{n})$ is defined by 
\begin{equation} \label{R1}
R(\bar{n}) =  m(n_1)^2n_1^2 - m(n_2)^2n_2^2 + \cdots - m(n_{2r})^2 n_{2r}^2.
\end{equation}

Recall that we assume \eqref{MAX}:
\[n^*_1 := \max (|n_1|, \dots, |n_{2r}|) > N.\]

\noi
We use $n^*_j$ to denote the $j$-th largest frequency in the absolute value.
Then, we have $n^*_2 \gtrsim N$ since $n_1 -n_2 + \cdots - n_{2r} = 0$.

In the following, we assume that $\D q$ satisfies both \eqref{L2} and \eqref{H1}.
Let $\mathbb{P}_{\geq N}$ be the Dirichlet projection onto the frequencies $\{|n| \geq N\}$.
Then, we have
\begin{align} \label{SL2}
\| \mathbb{P}_{\geq N} \D q \|_{L^2} \lesssim N^{-1} \|\D q\|_{H^1} \lesssim N^{-1}
\end{align}

\medskip
\noi
$\bullet$ {\bf Estimate on \eqref{HH3}:}
Let $\N$ and $\wt{\N}$ be of the form \eqref{Hamil6}
with frequencies $\{ n_j \}_{j = 1}^{2r}$ and $\{ \wt{n}_j \}_{j = 1}^{2\wt{r}}$.
In the following, we estimate
\begin{equation} \label{HHH3}
\sum_n m(n) \bigg(  \frac{\dd \N}{\dd q_n}(\D q) \frac{\dd \wt{\N}}{\dd \bar{q}_n}( q)
- \frac{\dd \N}{\dd q_n}( q) \frac{\dd \wt{\N}}{\dd \bar{q}_n}( \D q) \bigg),
\end{equation}

\noi
where
\begin{align} \label{FREQ}
n = n_2 - n_3 + \cdots + n_{2r}
= \wt{n}_1 -  \wt{n}_2 + \cdots + \wt{n}_{2\wt{r}-1}.
\end{align}

\noi
If $\max(n_1^*, \wt{n}_1^*) \leq  N$, then we have $\eqref{HHH3} = 0$.\footnote{Here, we 
abuse notation and set 
$n_1^* = \max( |n_2|, |n_3|, \dots, |n_{2r}|)$
and $\wt{n}_1^* = \max( |\wt{n}_1|, |\wt{n}_2|, \dots, |\wt{n}_{2\wt{r}-1}|)$.}
Hence, without loss of generality, assume $n_1^* > N$.
We consider only the first term in \eqref{HHH3}
since the second term can be estimated in a similar manner.
Now, we consider two cases: (a) $|n| \gtrsim N$, \ (b) $|n| \ll N$.

\medskip
\noi
$\circ$ Case (a): Suppose $|n| \gtrsim N$.
This implies $\wt{n}_1^* \gtrsim |n|\gtrsim N$.
By Cauchy-Schwarz inequality, we have
\begin{align} \label{A1}
\bigg|\sum_n m(n)   \frac{\dd \N}{\dd q_n}(\D q) \frac{\dd \wt{\N}}{\dd \bar{q}_n}( q)\bigg|
\leq \bigg\| \frac{\dd \N}{\dd q_n}(\D q) \bigg\|_{l^2_n}
 \bigg\|m(n) \frac{\dd \wt{\N}}{\dd \bar{q}_n}( q)\bigg\|_{l^2_n}.
\end{align}

First, let us consider the first factor.
By \eqref{SL2},
we have
\[ \| N m(n_1^*) q_{n_1^*} \|_{l^2_n} \lesssim 1.\]

\noi
By duality, we have
\begin{align} \label{A2}
\bigg\| \frac{\dd \N}{\dd q_n}(\D q) \bigg\|_{l^2_n}
= \sup_{\|p\|_{L^2} = 1} 
\sum_{n - n_2 + \cdots - n_{2r}= 0} c(\bar{n})\cdot p_n \cdot \cj{\D q}_{n_2} \cdots \cj{\D q}_{n_{2r}}
\lesssim N^{-1} \|\N\|,
\end{align}

\noi
where $p_n$ and $N \D q_{n_1^*}$ are the exceptional factors.

Next, consider the second factor in \eqref{A1}.
By the monotonicity of $m(\cdot)$ and \eqref{SL2},
we have
\[  \| N m(n) q_{\wt{n}_1^*} \|_{l^2_{\wt{n}_1^*}} 
\lesssim  \| N m(\wt{n}_1^*) q_{\wt{n}_1^*} \|_{l^2_{\wt{n}_1^*}} \lesssim 1.\]

\noi
By duality, we have
\begin{align} \label{A3}
\bigg\|m(n) \frac{\dd \wt{\N}}{\dd \bar{q}_n}( q)\bigg\|_{l^2_n}
& = \sup_{\|p\|_{L^2} = 1} 
\sum_{\wt{n}_1 - \wt{n}_2 + \cdots + \wt{n}_{2\wt{r}-1} -n = 0} 
c(\bar{n}) \cdot p_n \cdot m(n) q_{\wt{n}_1} \cj{q}_{\wt{n}_2} \cdots q_{\wt{n}_{2\wt{r}-1}} \notag\\
& \lesssim N^{-1} \|\wt{\N}\|,
\end{align}

\noi
where $p_n$ and $N m(n)q_{\wt{n}_1^*}$ are the exceptional factors.

\medskip
\noi
$\circ$ Case (b): Suppose $|n| \ll N$.
From \eqref{FREQ}, we have $n_2^* \gtrsim N$.
Thus, we have
\[ \| N m(n_1^*) q_{n_1^*} \|_{l^2}, \| N m(n_2^*) q_{n_2^*} \|_{l^2} \lesssim 1.\]

\noi
By Cauchy-Schwarz inequality, we have
\begin{align} \label{A5}
\bigg|\sum_n m(n)   \frac{\dd \N}{\dd q_n}(\D q) \frac{\dd \wt{\N}}{\dd \bar{q}_n}( q)\bigg|
\leq \bigg\| \frac{1}{\jb{n}}\frac{\dd \N}{\dd q_n}(\D q) \bigg\|_{l^2_n}
 \bigg\| \jb{n}m(n) \frac{\dd \wt{\N}}{\dd \bar{q}_n}( q)\bigg\|_{l^2_n}.
\end{align}

\noi
By duality, we have
\begin{align} \label{A6}
\bigg\| \frac{1}{\jb{n}} \frac{\dd \N}{\dd q_n}(\D q) \bigg\|_{l^2_n}
& = \sup_{\|p\|_{L^2} = 1} 
\sum_{n - n_2 + \cdots - n_{2r}= 0} c(\bar{n})\cdot \jb{n}^{-1} p_n \cdot \cj{\D q}_{n_2} \cdots \cj{\D q}_{n_{2r}} \notag \\
& \lesssim N^{-2} \|\N\|,
\end{align}

\noi
where $N \D q_{n_1^*}$ and $N \D q_{n_2^*}$ are the exceptional factors.
By Lemma \ref{LEM:POIS2} (a) and the monotonicity of $m(\cdot)$, 
the second factor in \eqref{A5} is bounded by $\|\wt{\N}\|$.

\medskip

From \eqref{A1}--\eqref{A6}, we obtain
\begin{equation} 
|\eqref{HH3}| \lesssim N^{-2} \|\N\|^2. \label{A4}
\end{equation}

\noi
Lastly, by writing $\N = \N_0 + \N_r$ and $\wt{\N} = \wt{\N}_0 + \wt{\N}_r$,
if $\N_r$ or $\wt{\N}_r$ appears in either the first or the second factor, then we have 
$|\eqref{HH3}| \lesssim N^{-C}$.

\medskip
\noi
$\bullet$ {\bf Estimate on \eqref{HH4}:}
Let $\eta(n^2) = m(n)^2 n^2$.
i.e. we have
\[ \eta(u) = \begin{cases}
u, & u \leq N^2 \\
N^{2(1-s)}u^s, & u \geq N^2.
\end{cases} \]

\noi
In particular, we have $\eta'(u) \lesssim \eta(u) /u$.

\medskip
\noi
$\circ$ $\N_0$-contribution:
Assume $n_1^* = |n_1|$.
Then, without loss of generality, we can assume
$n_2^* \sim |n_2|$
since $|D(\bar{n})| \leq N^\dl < (n^*_1)^\dl$
for small $\dl > 0$.
Thus, we have
\begin{align*}
n_1^2 = n_2^2 + O((n_3^*)^2 + N^\dl).
\end{align*}

\noi
By Mean Value Theorem, we have
\begin{align*}
|\eta(n_1^2) - \eta(n_2^2)| \lesssim m(n_1)^2 O((n_3^*)^2 + N^\dl).
\end{align*}

\noi
Thus, we have
\begin{align*}
|R(\bar{n})| & \lesssim m(n_1)^2 O((n_3^*)^2 + N^\dl) + m(n^*_3)^2 (n_3^*)^2\\
& \lesssim m(n_1)^2 O((n_3^*)^2 + N^\dl),
\end{align*}

\noi
where 
$R(\bar{n})$ is defined in \eqref{R1}.
Now, we consider two cases, depending on the size of  $n_3^*$:
(a) $n_3^* \lesssim N^\frac{\dl}{2}$,
(b) $n_3^* \gg  N^\frac{\dl}{2}$.

\medskip

\noi
$\circ$ Case (a): Suppose $n_3^* \lesssim  N^\frac{\dl}{2}$.
In this case, we have
$|R(\bar{n})| 
 \lesssim m(n_1)m(n_2)  N^{\dl}.$
Hence, from \eqref{SL2}, we have
\begin{align}
|\eqref{HH4}| & \lesssim N^{\dl}
\sum_{n_1 - n_2+ \cdots - n_{2r} = 0}
|c(\bar{n})|\cdot m(n_1)q_{n_1}\cdot m(n_2)\bar{q}_{n_2} \cdot q_{n_3}\cdots\bar{q}_{n_{2r}} \notag \\
& \lesssim N^{-2+\dl} \|\N_0\|. \label{BB1}
\end{align}

\noi
$\circ$ Case (b): Suppose $n_3^* \gg N^\frac{\dl}{2}$.
In this case, we have $n_4^*\sim n_3^*$ since $|D(\bar{n})| \ll (n_3^*)^2$.
Otherwise, if $n_4^*\ll n_3^*$, then
we would have
\begin{align} \label{B1}
(1+ o(1))(n_3^*)^2 = |(n_1 - n_2) (n_1 + n_2)|
= 
(1+ o(1))n_3^* (n_1^*+ n_2^*)
\end{align}

\noi
since $n_1$ and $n_2$ have the same sign in view of 
$|n_1 - n_2| = (1+ o(1))n_3^*.$
Then, it follows from \eqref{B1} that $n^*_3 = 0$,
which in turn implies
$n_1 = n_2$ and $n^*_3 = \cdots = n^*_{2r} = 0$.
In this case, we have $\eqref{HH4} = 0$ since $R(\bar{n}) = 0$.

Thus, we have
$|R(\bar{n})| 
 \lesssim m(n_1)m(n_2)  n_3^*n_4^*.$
By H\"older inequality and Sobolev embedding on the physical side, 
we have
\begin{align}
|\eqref{HH4}| & \lesssim 
\sum_{n_1 - n_2+ \cdots - n_{2r} = 0}
|c(\bar{n})|\cdot  m(n_1)q_{n_1}\cdot m(n_2)\bar{q}_{n_2} \cdot 
n_3^*\, n_4^*\, q_{n_3}\cdots\bar{q}_{n_{2r}} \notag \\
& \lesssim \|\mathbb{P}_{\gtrsim N} \D q\|_{H^{\frac{1}{2}+}}^2 \big\||\dx| \, q\big\|_{L^2}^2 
\|q\|_{H^{\frac{1}{2}+}}^{2r-4}  \notag \\
& \lesssim N^{-1+} \|\D q \|_{H^1}^{2r} \label{BB2}
\end{align}

\noi
since $\|\mathbb{P}_{\gtrsim N} \D q\|_{H^{\frac{1}{2}+}}
\lesssim N^{-\frac{1}{2}+}\|\D q\|_{H^1}$.

\medskip
From \eqref{BB1} and \eqref{BB2}, 
we obtain 
\begin{equation} 
| \eqref{HH4}| \lesssim N^{-1+}. \label{B3}
\end{equation}

\medskip
\noi
$\circ$ $\N_r$-contribution:
In this case, we use $|R(\bar{n})| \lesssim m(n_1^*)^2 (n_1^*)^2$.
Then, proceeding with $\| \D q \|_{H^1} \lesssim 1$, we have
\begin{equation} 
| \eqref{HH4}| \lesssim  \|\N_r\| <  N^{-C}. \label{B6}
\end{equation}

\medskip
\noi
$\bullet$ {\bf Estimate on \eqref{HH5}:}

\medskip
\noi
$\circ$ $\N_0$-contribution:
By proceeding with $|D(\bar{n}) | \leq N^\dl$ and \eqref{SL2} as before,
we obtain
\begin{equation} 
|\eqref{HH5}| \lesssim N^{-2+\dl}\|\N_0\|.  \label{C1}
\end{equation}

\medskip
\noi
$\circ$ $\N_r$-contribution:
In this case, we have $|D(\bar{n}) | \lesssim (n_1^*)^2$.
Hence, we obtain \eqref{B6}.

\subsection{Proof of Theorem \ref{thm2} (a)} \label{SEC:II}
Now, we are ready to put all the estimates together and prove Theorem \ref{thm2} (a).
Given $u_0 \in H^s$ with $s>1$, let 
\[ \D q_0 = \G^{-1} \D u_0.\]

\noi
Then, 
from \eqref{DD} and \eqref{H123}, we have
\[ \|\D q_0\|_{H^1}  \sim \| \D u_0\|_{H^1}
\leq \|u_0\|_{H^s} \sim_{s, u_0} 1.\]

From \eqref{A4}, \eqref{B3}, and \eqref{C1}, we have
\begin{equation} \label{D1}
\bigg| \frac{d}{dt} \HH(\D q) (t) \bigg| \lesssim N^{-1+}
\end{equation}

\noi
assuming 
\begin{equation}\label{D2}
 \|\D q (t) \|_{H^1} \lesssim 1.
\end{equation}

Now, fix $T>0$.
Suppose that \eqref{D2} holds true for $|t| \leq T$.
Then, from \eqref{D1}, we have 
\begin{equation} \label{D3}
\|\D q(t) \|_{H^1}^2 \sim 
\HH(\D q(t)) \leq  \HH(\D q(0)) + C T  N^{-1+}, \qquad |t| \leq T. 
\end{equation}

\noi
By choosing $N\sim T^{1+}$, 
we indeed have 
\begin{equation} \label{DD3}
\|\D q(t) \|_{H^1} \lesssim 1, \qquad |t| \leq T.
\end{equation}

Note that we performed the upside-down $I$-method on the transformed coordinates.
By  \eqref{DD}, \eqref{H123}, and \eqref{DD3}, we obtain
\[ \|u(t)\|_{H^s} \lesssim  N^{s-1} 
\|\D u(t) \|_{H^1}
\sim  N^{s-1} 
\|\D q(t) \|_{H^1}
\lesssim T^{(s-1)+}, \qquad |t| \leq T.\]

\noi
Therefore, we conclude that 
\[ \|u(t)\|_{H^s} \lesssim (1+|t|)^{(s-1)+}.\]

\noi
This completes the proof of Theorem \ref{thm2} (a).

\subsection{Improvement for $p \leq 2$: Theorem \ref{thm2} (b)}
\label{SEC:p23}

In this subsection, we briefly discuss 
how to improve the result when $p = 1, 2$.
The basic idea is to use the estimate due to Bourgain.
In  \cite{BO3, BO4}, Bourgain studied the quintic NLS,
where he used space-time estimates to obtain purely spatial estimates.
For this purpose, the $L^6$-Strichartz estimate \cite{BO1}:
\begin{equation} \label{L6}
\|e^{-i t \Dl} \phi\|_{L^6(\T^2)}
\lesssim C_N \|\phi\|_{L^2},
\quad
\supp \ft{\phi} \subset [-N, N]
\end{equation}

\noi 
plays a crucial role, where $C_N = \exp \big(C \frac{\log N}{\log \log N}\big) \ll N^{0+}$.
Then, one inductively proves estimates for Hamiltonians with higher order nonlinearity,
which appear in the process of the normal form reduction.
In the end, one obtains \cite[(5.13)]{BO4}:\footnote{One can indeed obtain this estimate
with $(n_3^*)^{0+}$ in place of $(n_1^*)^{0+}$, but it is not useful for our purpose.}
\begin{equation} \label{L666}
\max_{a \in \mathbb{Z}} \bigg| 
\sum_{\substack{n_1 - n_2 + \cdots - n_{2r} = 0\\ D(\bar{n})  = a}}
|c(\bar{n})| |q^{(1)}_{n_1}||q^{(2)}_{n_2}|\cdots |q^{(2r)}_{n_{2r}}| \bigg| 
\lesssim (n_1^*)^{0+} \prod_{j = 1}^{2r} \|q^{(j)}_{n_{j}}\|_{L^2}.
\end{equation}

%
%
%
%

For the cubic case ($p = 1$), one can basically repeat the same computation
to establish \eqref{bd1}, thanks to the $L^4$-Strichartz estimate \cite{ZYG}:
\begin{equation} \label{L4}
\|e^{-i t \Dl} \phi\|_{L^4(\T^2)}
\lesssim  \|\phi\|_{L^2}.
\end{equation}

\noi
Unlike \eqref{L6}, there is no derivative loss in \eqref{L4}.
However, one has a small derivative loss in the inductive steps,
causing the $+$ sign in \eqref{bd1}.
See (A.22) in \cite{BO3}.
As a conclusion, the estimate \eqref{L666} holds when $p = 1, 2$.

\medskip

Theorem \ref{thm2} (b) follows once we show
\begin{equation} \label{D4}
\bigg| \frac{d}{dt} \HH(\D q) (t) \bigg| \lesssim N^{-2+}.
\end{equation}

\noi
In view of \eqref{A4}, \eqref{BB1}, and \eqref{C1} with $\dl = 0+$, 
we only need to improve  Case (b) in Estimate on \eqref{HH4}.
By applying \eqref{L666} and \eqref{SL2}
in \eqref{BB2} (in place of H\"older inequality and Sobolev embedding),
we have
\begin{align*}
|\eqref{HH4}| & \lesssim 
 N^\dl   \max_{ |a| \leq N^\dl } 
\sum_{\substack{n_1 - n_2+ \cdots - n_{2r} = 0\\ D(\bar{n}) = a}}
|c(\bar{n})|\cdot  m(n_1)q_{n_1}\cdot m(n_2)\bar{q}_{n_2} \cdot 
n_3^*\, n_4^*\, q_{n_3}\cdots\bar{q}_{n_{2r}} \notag \\
& \lesssim N^\dl\|\mathbb{P}_{\gtrsim N} \D q\|_{H^{0+}}^2 \big\||\dx| \, q\big\|_{L^2}^2 
\|q\|_{L^2}^{2r-4}  \notag \\
& \lesssim N^{-2+\dl +} \|\D q \|_{H^1}^{2r}.
\end{align*}

\noi
Hence, \eqref{D4} follows and this proves Theorem \ref{thm2} (b).

\section{Cubic Case: Theorem \ref{thm2} (c)}

In this section, we consider the cubic case ($p = 1$)
and prove Theorem \ref{thm2} (c).
First, we explicitly compute  first few terms appearing in the normal form reduction
in Subsection \ref{SEC:Z1}.
See also Erdo\u gan-Zharnitsky \cite{EZ}.
Then, we establish improved estimates on those terms
and prove Theorem \ref{thm2} (c)
in Subsection \ref{SEC:Z2}.

\subsection{Normal form reduction: cubic NLS}  \label{SEC:Z1}

Let $H$ be as in \eqref{Hamil2} with $p = 1$.
i.e. we have
\begin{align} \label{ZH1}
H(q) &  = \sum_n n^2 |q_n|^2
+ \sum_{n_1 - n_2 + n_3 - n_4 =0} q_{n_1} \bar{q}_{n_2}  q_{n_3}\bar{q}_{n_4}
 =: H_0(q) + H_1(q). \notag
\end{align}

\noi
Now, divide $H_1$ into the resonant part $\RR$ and the non-resonant part $\N$,
depending on $D_1(\bar{n}) = 0$ or $\ne 0$,
where $D_1(\bar{n})$ is defined by
\[ D_1(\bar{n}) := n_1^2 - n_2^2 + n_3^2 - n_4^2 = -2 (n_1 - n_2) (n_3 - n_2).\]

\noi
We further split $\RR$ into two parts:
\begin{align} \label{Z3}
\RR =  \sum_{\substack{n_1 - n_2 + n_3 - n_4 =0\\D_1(\bar{n})= 0}} 
q_{n_1} \bar{q}_{n_2}  q_{n_3}\bar{q}_{n_4}
=  2\sum_{n_1}\sum_{n_3}  |q_{n_1}|^2 |q_{n_3}|^2 
- \sum_n |q_{n}|^4
=: \RR_1+\RR_2.
\end{align}

\noi
By the conservation of the $L^2$-norm,
we have 

\[\RR_1 = 2\mu \sum_n |q_n|^2, \]

\noi
where $ \mu = (2\pi)^{-1} \int |u|^2 dx.$
By a direct computation, one easily sees that 
$\{\RR_1, F\} = 0$ for smooth $F$ of the form \eqref{Hamil4}.

As the first step of the normal form reduction, 
define $F_1$ such that $\{H_0, F_1\} = - \N$.
i.e.
\begin{align} \label{Z4}
F_1 = \sum_{\substack{n_1 - n_2 + n_3 - n_4 =0\\n_2 \ne n_1, n_3}} 
\frac{q_{n_1} \bar{q}_{n_2}  q_{n_3}\bar{q}_{n_4}}
{-2 (n_1 - n_2) (n_3 - n_2)}.
\end{align}

\noi
Let $\G_1$ be the Lie transform associated with $F_1$.
Then, we have
\begin{align} \label{Z1}
H' : = H \circ \G_1
= H_0 + \RR_1 + \RR_2
+ \{ \RR_2, F_1\} + \tfrac{1}{2}\{\N, F_1\}  + \text{h.o.t.}
\end{align}

\noi
Here, we used the fact that $\{\RR_1, F_1\} = 0$
and $\{\N, F\} +\frac{1}{2}\{\{H_0, F_1\}, F_1\}
= \tfrac{1}{2}\{\N, F\}.$
From \eqref{Z3} and \eqref{Z4}, we have
\begin{equation} \label{ZR1}
\{ \RR_2, F_1\}
= 2i (\mathcal{I}_0 - \cj{\mathcal{I}_0}),
\end{equation}

\noi
where $\mathcal{I}_0$ is given by
\begin{align} \label{I0}
\mathcal{I}_0 
& = \sum_{\substack{n_1 - n_2 + n_3 - n_4 = 0\\
n_2 \ne n_1, n_3}}
\frac{q_{n_1}\bar{q}_{n_2}q_{n_3} |q_{n_4}|^2\bar{q}_{n_4}}
{(n_1 - n_2)(n_3 - n_2)}  \notag\\
&  = \sum_{\substack{n_1 - n_2 + n_3 - n_4 + n_5 - n_6= 0\\
n_2 \ne n_1, n_3\\
n_4 = n_5 = n_6}}
\frac{q_{n_1}\bar{q}_{n_2}q_{n_3} \bar{q}_{n_4}q_{n_5}\bar{q}_{n_6}}
{(n_1 - n_2)(n_3 - n_2)}.
\end{align}

Next, we introduce two more transformations 
to eliminate the ``non-resonant'' parts of $\{ \RR_2, F_1\}$ and 
$\tfrac{1}{2}\{\N, F_1\}$.
First, we divide them  
into the resonant parts (with $(\text{r})$)
and the non-resonant parts (with $(\text{nr})$),
\begin{align*}
\{ \RR_2, F_1\} & = \{ \RR_2, F_1\}^{(\text{r})} + \{ \RR_2, F_1\}^{(\text{nr})}\\
\{\N, F_1\} & = \{\N, F_1\}^{(\text{r})} + \{\N, F_1\}^{(\text{nr})},
\end{align*}

\noi
depending on 
\begin{equation} \label{ZD2}
|D_2(\bar{n})| \leq N^{\beta} \quad \text{or} \quad |D_2(\bar{n})| > N^{\beta}
\end{equation}

\noi
for some $\beta > 0$ (to be chosen later), where $D_2(\bar{n})$ is defined by
\[ D_2(\bar{n}) := n_1^2 - n_2^2 + n_3^2 - n_4^2 + n_5^2 - n_6^2.\]

\noi
Now,  define $F_2$ and $F_3$ such that 
\begin{align*}
 \{H_0, F_2\} & = \tfrac{1}{2} \{\N, F_1\}^{(\text{nr})}\\
\{H_0, F_3\} &  =\{ \RR_2, F_1\}^{(\text{nr})}
\end{align*}

\noi
i.e. we have 
\begin{equation} \label{ZD3}
F_2\sim \text{``}D_2^{-1}\{\N, F_1\}^{(\text{nr})}\, \text{''}
\quad \text{and} \quad
F_3\sim \text{``}D_2^{-1}\{ \RR_2, F_1\}^{(\text{nr})}\, \text{''}.
\end{equation}

\noi
Let $\G_2$ and $\G_3$ be the Lie transforms associated with $F_2$ and $F_3$.
Then, from \eqref{Z1}, we have
\begin{align} \label{Z2}
H'' : = H \circ \G_1 \circ \G_2 \circ \G_3
= H_0 + \RR_1 + \RR_2
+ \{ \RR_2, F_1\}^{(\text{r})} + \tfrac{1}{2}\{\N, F_1\}^{(\text{r})}  + \text{h.o.t.}
\end{align}

\noi
From \eqref{ZR1} and \eqref{Z4}, we have
\begin{align}
\{ \RR_2, F_1\}^{(\text{r})}
& = 2i (\mathcal{I}_1 - \cj{\mathcal{I}_1}),\\
\tfrac{1}{2}\{\N, F_1\}^{(\text{r})}
& = 2i (\mathcal{I}_2 - \cj{\mathcal{I}_2}),
\end{align}

\noi
where  $\mathcal{I}_1$ is the resonant part (i.e. $|D_2(\bar{n})| \leq N^{\beta}$)
of $\mathcal{I}_0$ defined in \eqref{I0}
and 
$\mathcal{I}_2$ is given by
\begin{equation} \label{I2}
\mathcal{I}_2 = \sum_{\substack{n_1 - n_2 + n_3 - n_4 + n_5 - n_6= 0\\
n_2 \ne n_1, n_3\\n_5 \ne n_4, n_6\\ |D_2(\bar{n})| \leq N^{\beta}}}
\frac{q_{n_1}\bar{q}_{n_2}q_{n_3} \bar{q}_{n_4}q_{n_5}\bar{q}_{n_6}}
{(n_1 - n_2)(n_3 - n_2)}.
\end{equation}

In the next subsection, we estimate the terms
$\RR_1$, $\RR_2$, $\{ \RR_2, F_1\}^{(\text{r})}$, and $\tfrac{1}{2}\{\N, F_1\}^{(\text{r})}$
appearing in \eqref{Z2}.
Also,  note that all the higher order terms in \eqref{Z2}
are Poisson-bracketed with $F_2$ or $F_3$ at least once.
i.e. they have an extra decay of $|D_2|^{-1} < N^{-\beta}$
from \eqref{ZD2} and \eqref{ZD3}.

After this point, 
we perform the (usual) normal form reductions 
(as in Section \ref{SEC:NORMAL})
on the higher order terms in \eqref{Z2}.
In particular,  we use \eqref{N1} and \eqref{N2}
with $\dl = 0+$
to distinguish the resonant and non-resonant terms.
In the process, we construct Hamiltonians $F$ of the form \eqref{Hamil8}
to eliminate the non-resonant parts of the higher order terms in \eqref{Z2}.
For such $F$, it follows from the observation in the previous paragraph
that  $c(\bar{n})$ in \eqref{Hamil8} is small,
i.e. $|c(\bar{n})| < N^{-\beta}$.
After a finite number of iterations, 
\eqref{Z2} is reduced to 
\begin{align} \label{Z5}
\HH  = \wt{H}_0  + \RR_2
+ \{ \RR_2, F_1\}^{(\text{r})} + \tfrac{1}{2}\{\N, F_1\}^{(\text{r})}  + 
\underbrace{\N_0+ \N_r}_{= \text{ h.o.t.}}
=: \wt{H}_0 + \wt{\N},
\end{align}

\noi
where $\wt{H}_0$ is the new quadratic part defined by
\[ \wt{H}_0 : = H_0 + \RR_1 = \sum_n (n^2 + 2\mu) |q_n|^2\]

\noi 
and the higher order terms have an extra factor of $N^{-\beta}$.
(Compare this with \eqref{Hamil9}.)

\subsection{Improved estimates} \label{SEC:Z2}
In this subsection, we prove Theorem  \ref{thm2} (c)
by establishing improved estimates
for all the terms in \eqref{Z5}.
Differentiating \eqref{Z5} in time, we obtain 
\begin{align}
\frac{d}{dt} \HH(\D q) & = \frac{\dd \HH}{\dd q}(\D q)\cdot \D q_t 
+ \frac{\dd \HH}{\dd \bar{q}}(\D q)\cdot \cj{\D q}_t  \notag \\
& = i \sum_n m(n)^2 (n^2+ 2\mu) \bigg( \bar{q}_n \frac{\dd \wt{\N}}{\dd \bar{q}_n}(q)
 - q_n \frac{\dd \wt{\N}}{\dd q_n}(q) \bigg)  \label{ZHH1} \\
& + i \sum_n m(n) (n^2+ 2\mu) \bigg( q_n \frac{\dd \wt{\N}}{\dd q_n}(\D q)
- \bar{q}_n \frac{\dd \wt{\N}}{\dd \bar{q}_n}(\D q) \bigg)  \label{ZHH2}\\
& + i \sum_n m(n) \bigg(  \frac{\dd \wt{\N}}{\dd q_n}(\D q) \frac{\dd \wt{\N}}{\dd \bar{q}_n}( q)
- \frac{\dd \wt{\N}}{\dd q_n}( q) \frac{\dd \wt{\N}}{\dd \bar{q}_n}( \D q) \bigg).  \label{ZHH3}
\end{align}

\noi
In the following, we simply use $\jb{n}$ for  $\jb{n}_{\mu} := (n^2+ 2\mu)^\frac{1}{2}$
since $\mu$ is a fixed constant thanks to the $L^2$-conservation.

First, note that Theorem \ref{thm2} (c) follows
once we prove
\begin{equation} 
\bigg| \frac{d}{dt} \HH(\D q) (t) \bigg|
\leq |\eqref{ZHH1}|+ |\eqref{ZHH2}|+ |\eqref{ZHH3}| \lesssim N^{-\frac{9}{4}+}.\label{ZZ1}
\end{equation}

\noi
Also, note that the terms  \eqref{ZHH1}--\eqref{ZHH3} are basically the same as 
\eqref{HH1}--\eqref{HH3}.
Thus, by comparing \eqref{D4} and \eqref{ZZ1}, 
it suffices to show that there is an additional decay of $N^{-\frac{1}{4}+}$
in this case. 

As mentioned at the end of the last subsection, 
all the higher order terms in \eqref{Z5}
have an extra decay of $|D_2|^{-1} < N^{-\beta}$.
Hence, by repeating the argument in Subsection \ref{SEC:p23}
with this extra decay of $N^{-\beta}$,
we have
\begin{equation}
|\eqref{ZHH1}|+ |\eqref{ZHH2}|+ |\eqref{ZHH3}| \lesssim N^{-2-\beta+}.\label{WZ0}
\end{equation}

\noi
Moreover, if either of $\N$ or $\wt{\N}$ in \eqref{HHH3}, say $\N$, 
is one of the higher order terms, 
then, we also gain an extra decay of $N^{-\beta}$ from $\N$,
and thus \eqref{WZ0} holds.

\medskip
Therefore, we only consider 
the contributions from  $\RR_2$, $\mathcal{I}_1$ for  $\{ \RR_2, F_1\}^{(\text{r})}$, and 
$\mathcal{I}_2$ for $\tfrac{1}{2}\{\N, F_1\}^{(\text{r})}$ in the following.
Recall that the main idea in Subsections \ref{SEC:NONLIN1} and \ref{SEC:p23}
is to identify large frequencies and apply \eqref{SL2} 
to gain a negative power of $N$.
In particular, it follows from \eqref{L666} and \eqref{SL2} that 
for each large frequency $\, \gtrsim N$,
we basically gain a power of $N^{-1}$.

\medskip
First, let us use \eqref{L666} to establish preliminary estimates on 
$\RR_2$, $\mathcal{I}_1$, and  $\mathcal{I}_2$,
assuming \eqref{MAX}:
\[n^*_1 := \max (|n_1|, \dots, |n_{2r}|) > N.\]

\medskip
\noi
$\bullet$ {\bf (i) On $\RR_2$:} By writing $\RR_2$ in the form \eqref{Hamil6}, we have
\begin{align} 
\RR_2 (q)=  - \sum_n |q_{n}|^4
= -\sum_{\substack{n_1 - n_2 + n_3 - n_4 =0\\n_1 = n_2 = n_3 = n_4}} 
q_{n_1} \bar{q}_{n_2}  q_{n_3}\bar{q}_{n_4}.
\end{align}

\noi
By \eqref{MAX}, we have $|n_j| >N$, $j = 1, \cdots, 4.$
Then, from \eqref{L666} and \eqref{SL2}, 
we have
\begin{align} \label{W1}
|\RR_2(q)| \lesssim  \|\mathbb{P}_{\geq N} q\|_{H^{0+}}^4
\leq N^{-4+} \| q\|_{H^{1}}^4.
\end{align}

\medskip
\noi
$\bullet$ {\bf (ii) On $\mathcal{I}_1$:}
From \eqref{I0}, we have
\begin{align}
\mathcal{I}_1(q)
&  = \sum_{\substack{n_1 - n_2 + n_3 - n_4 + n_5 - n_6= 0\\
n_2 \ne n_1, n_3\\
n_4 = n_5 = n_6\\ |D_2(\bar{n})| \leq N^\beta }}
\frac{q_{n_1}\bar{q}_{n_2}q_{n_3} \bar{q}_{n_4}q_{n_5}\bar{q}_{n_6}}
{(n_1 - n_2)(n_3 - n_2)}.
\end{align}

\noi
If $|n_4| \gtrsim N$, then we have at least three large frequencies ( $\gtrsim N$.)
Thus, from \eqref{L666} and \eqref{SL2}, 
we have
\begin{align} \label{W2}
|\mathcal{I}_1(q)| 
\lesssim N^\beta \Big(\prod_{j = 1}^3 \|q_{n_j}\|_{l^2} \Big)
\|\mathbb{P}_{\geq N} q\|_{H^{0+}}^3
\leq N^{-3+\beta+} \| q\|_{H^{1}}^6.
\end{align}

\noi
If $|n_4| \ll N$, then 
there are at least two frequencies among $n_1, n_2, n_3$
of size $\gtrsim N$.
If $\min (|n_1|, |n_2|, |n_3|) \gtrsim N$,
then we have 
$|\mathcal{I}_1(q)| \leq N^{-3+\beta+} \| q_n\|_{H^{1}}^6$
as in \eqref{W2}.
If $\min (|n_1|, |n_2|, |n_3|) \ll N$,
then we have
\[|(n_1 - n_2)(n_3 - n_2)| \gtrsim N.\]

\noi
Hence, we have
\begin{align} \label{W3}
|\mathcal{I}_1(q)| 
\lesssim N^{-1+\beta} \|\mathbb{P}_{\gtrsim N} q\|_{H^{0+}}^2 \|q_{n}\|^4_{l^2} 
\leq N^{-3+\beta+} \| q\|_{H^{1}}^6.
\end{align}

\medskip
\noi
$\bullet$ {\bf (iii) On $\mathcal{I}_2$:}
We have $n_1^* \geq n_2^* \gtrsim N$.
If $n_3^* \gtrsim N $, then we have
\begin{align} \label{W4}
|\mathcal{I}_2(q)| 
\lesssim N^\beta \|\mathbb{P}_{\gtrsim N} q\|_{H^{0+}}^3
\|q_n\|_{l^2}^3
\leq N^{-3+\beta+} \| q\|_{H^{1}}^6.
\end{align}

\noi
Hence, suppose $n_3^* \ll N$ in the following.

\medskip
\noi
$\circ$ Case (iii.1): 
Suppose $\max (|n_1|, |n_2|, |n_3|) \gtrsim N$.
Then we have $|(n_1 - n_2)(n_3 - n_2)| \gtrsim N$.
Hence, $|\mathcal{I}_2(q)| 
\lesssim  N^{-3+\beta+} \| q\|_{H^{1}}^6$ as before.

\medskip
\noi
$\circ$ Case (iii.2): 
Suppose $\max (|n_1|, |n_2|, |n_3|) \ll N$.
Let $\beta < 1$.
Then, $|D_2(\bar{n})| \leq N^\beta$ implies 
 $|n_5| \geq n_2^*$.
Otherwise, i.e., if $|n_5| \leq n_3^* \ll N$,
then we would have $|n_4|, |n_6| \gtrsim N$,
and thus
\[ -n_4^2 - n_6^2 = D(\bar{n}) + o(N^2) = o(N^2).\]

\noi
This is clearly a contradiction.
Hence, we have $|n_5| \geq n_2^*$.
Without loss of generality assume $|n_4| \geq |n_6|$.
i.e. $\{|n_4|, |n_5|\} = \{n_1^*, n_2^*\}$.

$\diamond$ Subcase (iii.2.a):
Suppose $n_3^* \lesssim N^{\frac{1}{2}-}$.
Then, write $n_4$ as 
$n_4 = n_5 + m $,
where $m = O(N^{\frac{1}{2}-})$ and $m \ne 0$.  (Recall $n_4 \ne n_5$.) 
Then, we have
\begin{equation*}
|D_2(\bar{n})| = |n_4^2- n_5^2 + O(N^{1-})|
= |2m n_5 + O(N^{1-})| \gtrsim |m n_5| \gtrsim N.
\end{equation*}

\noi
This contradicts with $|D_2(\bar{n})| \leq N^\beta \ll N$.

$\diamond$ Subcase (iii.2.b):
Suppose $n_3^* \gg N^{\frac{1}{2}-}$.
Then, we have $|D_2(\bar{n})| \leq N^\beta \leq N^{1-} \ll (n_3^*)^2$.
This in turn implies
$n_4^* \sim n_3^*$ as in Case (b) of Estimate on \eqref{HH1}
in Subsection \ref{SEC:NONLIN1}.
Thus, we have
\begin{align} \label{W5}
|\mathcal{I}_2(q)| 
\lesssim N^\beta  \|\mathbb{P}_{\gtrsim N} q\|_{H^{0+}}^2
\|\mathbb{P}_{\gtrsim N^{\frac{1}{2}-}} q\|_{L^2}^2
\|q\|_{L^2}^2 
\leq N^{-3+\beta+} \| q\|_{H^{1}}^6.
\end{align}

\noi
Therefore, we have
$|\mathcal{I}_2(q)| 
\lesssim 
 N^{-3+\beta+} \| q\|_{H^{1}}^6$
as long as $\beta < 1$.

\medskip

In the following, we estimate the contributions
from $\RR_2$, $\mathcal{I}_1$, and $\mathcal{I}_2$
for \eqref{ZHH1}, \eqref{ZHH2}, and \eqref{ZHH3},
assuming $\beta < 1$.

\medskip
\noi
$\bullet$ {\bf Estimate on \eqref{ZHH2}:} 
Since $\sum_{j = 1}^{2r} (-1)^{j+1} (n_j^2+2\mu) = D(\bar{n})$,
we can rewrite \eqref{ZHH2} in the form \eqref{HH5}.
First, note that there is no contribution from $\RR_2$
since $D_1(\bar{n}) = 0$.
From \eqref{W1}--\eqref{W5}, we have
\begin{align*}
|\mathcal{I}_1|, \, |\mathcal{I}_2| \lesssim N^{\beta} N^{-3+\beta+}
=  N^{-3+2\beta+}.
\end{align*}

\noi
Therefore, we have
\begin{equation} 
|\eqref{ZHH2}|\lesssim N^{-3+2\beta+}. \label{WZ1} 
\end{equation}

\medskip
\noi
$\bullet$ {\bf Estimate on \eqref{ZHH1}:} 
First, we rewrite \eqref{ZHH1} as before.
\begin{align*} 
\eqref{ZHH1}
& = - \sum_{n_1 - n_2+ \cdots - n_{2r} = 0}
c(\bar{n}) \wt{R}(\bar{n})
q_{n_1}\bar{q}_{n_2}\cdots\bar{q}_{n_{2r}} 
\end{align*}

\noi
where $\wt{R}(\bar{n})$ is defined by 
\begin{equation} \label{R2}
\wt{R}(\bar{n}) =  m(n_1)^2\jb{n_1}^2 - m(n_2)^2\jb{n_2}^2 + \cdots - m(n_{2r})^2 \jb{n_{2r}}^2.
\end{equation}

\noi
Once again, there is no contribution from $\RR_2$
since $\wt{R}(\bar{n}) = 0$ when $n_1 = \cdots = n_4$.

In the following, we estimate 
the contribution from
$\mathcal{I}_1$ and $\mathcal{I}_2$
on \eqref{ZHH1}.
By repeating the computation in Estimate on \eqref{HH1}
in Subsection \ref{SEC:NONLIN1}, we have
\begin{equation*}
|\wt{R}(\bar{n})| \lesssim m(n_1^*)^2O((n_3^*)^2+N^\beta).
\end{equation*}

\medskip
\noi
$\circ$ Case (a): Suppose $n_3^* \lesssim N^\frac{\beta}{2}.$
In this case, we have 
$|\wt{R}(\bar{n})| \lesssim m(n_1^*) m(n_2^*) N^\beta$.
Hence, from \eqref{W1}--\eqref{W5},
the contribution from
$\mathcal{I}_1$ and $\mathcal{I}_2$
can be estimated as 
\begin{align} \label{ZW2}
|\eqref{ZHH1}| \lesssim  N^{-3+2\beta+}.
\end{align}

\medskip
\noi
$\circ$ Case (b): Suppose $n_3^* \gg N^\frac{\beta}{2}.$
In this case, we have $n_4^* \sim n_3^*$
as in Subsection \ref{SEC:NONLIN1}.
Hence, we have $|\wt{R}(\bar{n})| \lesssim m(n_1^*) m(n_2^*) n_3^*n_4^*$.

\medskip

First, we estimate the contribution from $\mathcal{I}_1$.

\medskip
$\diamond$ Subcase (b.1):
Suppose $ |n_4| \gtrsim N$.
This implies that 
$\max(n_1, n_2, n_3) \geq n_4^* \sim n_3^* \gtrsim N$.

If $n_5^* \ll N$, then we have 
$\med(n_1, n_2, n_3) = n_5^* \ll N$
and thus
$|(n_1 - n_2)(n_3 - n_2)| \gtrsim N$.
Hence, we have
\begin{align} \label{ZW3}
|\eqref{ZHH1}| 
& \lesssim N^{-1+\beta}
\|\mathbb{P}_{\gtrsim N} \D q\|_{H^{0+}}^2
\Big(\prod_{j = 3}^4 \|n_j^* q_{n_j^*}\|_{l^2}\Big)
\|q\|_{L^2}^2
\leq N^{-3+\beta+} \| \D q\|_{H^{1}}^6  \notag \\
& \lesssim N^{-3+\beta+}.
\end{align}

\noi
Otherwise, i.e. if $n_5^* \gtrsim N$, 
then  we have
\begin{align} \label{ZW4}
|\eqref{ZHH1}|  
& \lesssim N^{\beta}
\|\mathbb{P}_{\geq N} \D q\|_{H^{0+}}^2
\Big(\prod_{j = 3}^4 \|n_j^* q_{n_j^*}\|_{l^2}\Big)
\|\mathbb{P}_{\geq N}  q\|_{H^{0+}}
\|q\|_{L^2} 
 \lesssim N^{-3+\beta+}.
\end{align}

\medskip
$\diamond$ Subcase (b.2):
Suppose $|n_4|\ll N$.
This implies $n_3^* \sim n_4^* \ll N$.
Hence, we have $|(n_1 - n_2)(n_3 - n_2)| \gtrsim N$
and \eqref{ZW3} holds in this case.

\medskip

Next, we estimate the contribution from $\mathcal{I}_2$.

\medskip
$\diamond$ Subcase (b.3):
Suppose $n_3^* \gtrsim N$.
We have $n_4^* \gtrsim N$ since $n_4^* \sim n_3^*$.
Then, as in Subcase (b.1),  we obtain 
\eqref{ZW3} or \eqref{ZW4}, depending on the size of $n_5^*$.

\medskip
$\diamond$ Subcase (b.4):
Suppose $n_3^* \ll N$.
If $\max (n_1, n_2, n_3) \gtrsim N$,
then we have $|(n_1 - n_2)(n_3 - n_2)| \gtrsim N$.
Hence, $|\eqref{ZHH1}|  
\lesssim  N^{-3+\beta+} $ as in \eqref{ZW3}. 

Now, suppose 
$n_3^* \ll N$
 and $\max (n_1, n_2, n_3) \ll N$.
Then, 
as in  Subcase (iii.2.a) for the preliminary estimate on $\mathcal{I}_2$, 
the case $n_3^* \lesssim N^{\frac{1}{2}-}$ can not occur.
Hence, we have $n_3^* \gg N^{\frac{1}{2}-}$.

\medskip
$\diamond$ Subsubcase (b.4.i):
Suppose $n_5^* \ll n_3^*$.
It follows from $\max (n_1, n_2, n_3) \ll N \lesssim n_2^*$
that either
(a) two frequencies of $|n_1|, |n_2|, |n_3|$
are $O(n_3^*)$, and the other one is $o(n_3^*)$,
or  (b) one frequency of $|n_1|, |n_2|, |n_3|$
is $O(n_3^*)$, and the other two are $o(n_3^*)$.
In either case, we have
\[|(n_1 - n_2)(n_3 - n_2)| \gtrsim  O(n_3^*)
\gg N^{\frac{1}{2}-}.\]

\noi
Hence, we have
\begin{align} \label{ZW5}
|\eqref{ZHH1}|  
\lesssim N^{-\frac{1}{2}+\beta+}
\|\mathbb{P}_{\gtrsim N} \D q\|_{H^{0+}}^2
\Big(\prod_{j = 3}^4 \|n_j^* q_{n_j^*}\|_{l^2}\Big)
\|q\|_{L^2}^2
\lesssim N^{-\frac{5}{2}+\beta+}. 
\end{align}

\medskip
$\diamond$ Subsubcase (b.4.ii):
Suppose $n_5^* \sim n_3^*$.
In this case, we have
\begin{align} \label{ZW6}
|\eqref{ZHH1}|   
\lesssim N^{\beta}
\|\mathbb{P}_{\geq N} \D q\|_{H^{0+}}^2
\Big(\prod_{j = 3}^4 \|n_j^* q_{n_j^*}\|_{l^2}\Big)
\|\mathbb{P}_{\geq N^{\frac{1}{2}-}} q\|_{H^{0+}}
\|q\|_{L^2}
\lesssim N^{-\frac{5}{2}+\beta+}.
\end{align}

\medskip

\noi
From \eqref{ZW2}--\eqref{ZW6}, we conclude
\begin{equation} 
|\eqref{ZHH1}|\lesssim N^{-\frac{5}{2}+\beta+}. \label{WZ2} 
\end{equation}

\medskip
\noi
$\bullet$ {\bf Estimate on \eqref{ZHH3}:} 
We follow the argument in Estimate on \eqref{HH3}
in Subsection \ref{SEC:NONLIN1}.
It suffices to estimate
$\sum_n m(n)   \frac{\dd \N_1}{\dd q_n}(\D q) \frac{\dd \N_2}{\dd \bar{q}_n}( q)$.
where $\N_1$ and $\N_2$ are either $\RR_2$, $\mathcal{I}_1$, or $\mathcal{I}_2$
with frequencies $\{ n_j \}_{j = 1}^{2r}$ and $\{ \wt{n}_j \}_{j = 1}^{2\wt{r}}$.

\medskip
\noi
$\circ$ Case (a): Suppose $|n| \gtrsim N.$
In this case, we have $\wt{n}_1^* \gtrsim |n| \gtrsim N$.

If $\RR_2$ appears in one of the factors, say $\N_1 = \RR_2$, 
then, by duality with \eqref{L666}
(note $D_1(\bar{n}) = 0$), we have
\begin{align*} 
\bigg\| \frac{\dd \RR_2}{\dd q_n}(\D q) \bigg\|_{l^2_n}
& = \sup_{\|p\|_{L^2} = 1} 
\sum_{\substack{n - n_2 + n_3 - n_4= 0\\n_2 = n_3 = n_4 = n}} 
 p_n \cdot \cj{\D q}_{n_2}\D q_{n_3}  \cj{\D q}_{n_{4}}\\
& \lesssim  \|\mathbb{P}_{\geq N} q\|_{H^{0+}}^3
\lesssim N^{-3+}.
\end{align*}

\noi
By Cauchy-Schwarz inequality with \eqref{A3}, we have
\begin{align} \label{ZA1}
\bigg|\sum_n m(n)   \frac{\dd \RR_2}{\dd q_n}(\D q)  \frac{\dd \N_2}{\dd \bar{q}_n}( q)\bigg|
 & \leq \bigg\| \frac{\dd \RR_2}{\dd q_n}(\D q) \bigg\|_{l^2_n}
 \bigg\|m(n) \frac{\dd \N_2}{\dd \bar{q}_n}( q)\bigg\|_{l^2_n} \notag \\
&   \lesssim N^{-3+} N^{-1} \|\N_2\|
\lesssim N^{-4+}.
\end{align}

Hence, we assume that both $\N_1$ and $\N_2$ are 
either $\mathcal{I}_1$ or $\mathcal{I}_2$.
Then, by applying Cauchy-Schwarz inequality, duality, 
and the preliminary estimates on $\mathcal{I}_1$ or $\mathcal{I}_2$
in (ii)--(iii) on each factor,
we obtain
\begin{equation}
|\eqref{ZHH3}| \lesssim N^{-2+\beta+}N^{-2+\beta+}
= N^{-4+2\beta+}. \label{ZW7}
\end{equation}

\noi
Note that 
we gain only $N^{-2+\beta+}$ from each factor, instead of $N^{-3+\beta+}$ as in (ii)--(iii).
This is due to the fact that a duality variable $p$ is only in $L^2$
and thus we can not gain an extra power of $N$
through \eqref{SL2}.

\medskip
\noi
$\circ$ Case (b): Suppose $|n| \ll N.$
Then, we have either $n_1^*, n_2^* \gtrsim N$
or $\wt{n}_1^*, \wt{n}_2^* \gtrsim N$.\footnote{Recall 
a slight abuse of notation for $n_j^*$ and $\wt{n}_j^*$.
See Estimate on \eqref{HH3}
in Subsection \ref{SEC:NONLIN1}.}
Suppose $n_1^*, n_2^* \gtrsim N$.
Then, we can use the preliminary estimates on $\RR_2$, $\mathcal{I}_1$, or $\mathcal{I}_2$
in (i)--(iii)
for the first factor (after duality)
and use Lemma \ref{LEM:POIS2} (a) for the second factor:
\begin{equation} 
|\eqref{ZHH3}| \lesssim  
\bigg\| \frac{1}{\jb{n}}\frac{\dd \N_1}{\dd q_n}(\D q) \bigg\|_{l^2_n}
 \bigg\|\jb{n}m(n) \frac{\dd \N_2}{\dd \bar{q}_n}( q)\bigg\|_{l^2_n} 
 \lesssim 
N^{-3+\beta+}. \label{ZW8}
\end{equation}

\medskip
\noi
From \eqref{ZW7} and \eqref{ZW8}, we conclude
\begin{equation} 
|\eqref{ZHH3}| \lesssim \max(N^{-4+2\beta+}, N^{-3+\beta+}) = N^{-3+\beta+}. \label{WZ3}
\end{equation}

\noi
for $\beta < 1$.

\medskip
Putting together all the estimates
from \eqref{WZ0}, \eqref{WZ1},\eqref{WZ2}, and \eqref{WZ3}, 
we have
\begin{align*}
|\eqref{ZHH1}|+ |\eqref{ZHH2}|+ |\eqref{ZHH3}| 
\lesssim \max( N^{-2-\beta+},N^{-3+2\beta+}, N^{-\frac{5}{2}+\beta+})
\end{align*}

\noi
By choosing $\beta = \frac{1}{4}$, 
\eqref{ZZ1} follows.
This completes the proof of Theorem \ref{thm2} (c).

\appendix

\section{On Theorem \ref{thm1}}
In \cite{BO3}, Bourgain presented details for the quintic nonlinearity ($ p = 2$.)
After the normal form reduction, 
\eqref{L666} was enough to conclude the result.
For the cubic case ($p = 1$),
Theorem \ref{thm1} (a) follows once we note that 
\eqref{L666} still holds in this case, as discussed in subsection \ref{SEC:p23}.

For $p \geq3$, there is no Strichartz estimate available in the periodic setting, 
and thus we need to rely on Sobolev inequality.
However, we can still perform the normal form reduction as in Section \ref{SEC:NORMAL}
(with $K = T^\dl$) since 
both \eqref{L2} and \eqref{H1} are  satisfied for all $t \in \R$
thanks to the $L^2$-conservation
and the conservation of the defocusing Hamiltonian.
Hence, we can proceed as in \cite{BO3}.

Let 
$I_s(q) = \|q\|_{H^s}^2 = \sum_n |n|^{2s} |q_n|^2$.
Then, after the normal form reduction, we have (see (A.29) in \cite{BO3})
\begin{equation} \label{E1}
\dt I_s \lesssim 
\bigg|\sum_{\substack{n_1 - n_2 + \cdots - n_{2r}=0\\ |D(\bar{n})|< K}} 
c(\bar{n}) D_s(\bar{n})
q_{n_1} \bar{q}_{n_2} \cdots q_{n_{2r-1}}\bar{q}_{n_{2r}}\bigg|
\end{equation}

\noi
where
$D_s(\bar{n}) :=  \sum_j  (-1)^j |n_j|^{2s}.$
\noi
By Lemma on p.1355 in \cite{BO3}, we have
\[|D_s(\bar{n})| \lesssim (n_1^*)^{2(s-1)} (n_3^* n_4^* + K).\]

\noi
(Note a typo in the statement (A.32) in \cite{BO3}.)
Assume that $n_j^* = |n_j|$, $j = 1, \dots, 4$.
Moreover, assume $K \leq |n_3| |n_4|$.
Then, we have 
\begin{align} 
|\eqref{E1}| \lesssim 
\sum_{\substack{n_1 - n_2 + \cdots - n_{2r}=0\\ |D(\bar{n})|< K}} 
|c(\bar{n})| 
\bigg(\prod_{j = 1}^2  |n_j|^{(s-1)}  |q_{n_j}|\bigg)
\bigg(\prod_{j = 3}^4 |n_j| |q_{n_j}|\bigg) 
\bigg(\prod_{j = 5}^{2r} |q_{n_j}|\bigg). \label{E2}
\end{align}

\noi
For the quintic case in \cite{BO3}, 
it is at this point (see (A.37)--(A.38) in \cite{BO3})
that the space-time estimate \cite[(A.18)]{BO3} was used.
As mentioned above, (A.18) in \cite{BO3} follows from from the $L^6$-Strichartz estimate \eqref{L6}.
For $p \geq 3$, we do not have such an estimate.
Thus, we simply proceed by H\"older inequality and Sobolev embedding on the physical side, and obtain
\begin{align} 
|\eqref{E2}| \lesssim \|q\|_{H^{s- \frac{1}{2}+}}^2 \|q\|_{H^1}^2 \|q\|_{H^{\frac{1}{2}+}}^{2r-4}
\lesssim I_s^{1-\theta}\|q\|_{H^1}^{2r-2 + 2\theta}, \label{E3}
\end{align}

\noi
where in the last step we used interpolation:
$ \|q\|_{H^{s- \frac{1}{2}+}}
\leq \|q\|_{H^s}^{1-\theta} \|q\|_{H^1}^\theta$
with 
\begin{equation} \label{E4}
\theta = \frac{1}{2(s-1)+}.
\end{equation}

\noi
If $ |n_3| |n_4| \leq K = T^\dl$, then
we obtain \eqref{E3} with an extra factor of $K = T^\dl$.
In view of the uniform bound on the $H^1$-norm on solutions, we obtain
\[ \dt I_s \lesssim T^\dl I_s^{1-\theta} \quad \LRA
\quad \dt (I_s^\theta) \lesssim T^{\dl}.\]

\noi
Hence, we obtain $I_s(t) \lesssim T^{\frac{1+\dl}{\theta}}  = T^{2(s-1)+}$ for $|t| \leq T$ (with $\dl = 0+$.)
This proves Theorem \ref{thm1} (b).

\medskip
\noi
{\bf Acknowledgment:}
J.C. and T.O. would like to thank Alessandro Selvitella 
for a lecture on the classical theorem of the Birkhoff normal form.

\end{document}